\documentclass{article}

\usepackage{amscd,amssymb,verbatim,graphicx,stmaryrd,amsthm,color} 
\usepackage[colorlinks]{hyperref}

\input xy    
\xyoption{curve}  
\xyoption{arc} 
\xyoption{frame} 
\xyoption{tips} 
\xyoption{arrow} 
\xyoption{color} 

\usepackage{mathpazo} 

\newtheorem{theorem}{Theorem}[section]  
\newtheorem{lemma}[theorem]{Lemma}  
\newtheorem{proposition}[theorem]{Proposition}
\newtheorem{corollary}[theorem]{Corollary}

\newtheorem{remark}[theorem]{Remark}
\newtheorem{example}[theorem]{Example}

\usepackage{blindtext}
\usepackage{changepage}
\newtheoremstyle{indented}{3pt}{3pt}{}{}{\bfseries}{.}{.5em}{}
\theoremstyle{indented}
\newtheorem{innercustomthm}{}
\newenvironment{customthmQuant}[1]{\renewcommand\theinnercustomthm{#1}\innercustomthm}{\endinnercustomthm}

\newcounter{theexercise} 

{\begin{list}{}
         {\setlength{\leftmargin}{#1}}
         \item[]}
{\end{list}}

\newcommand{\End}{\mathrm{End}}
\newcommand{\sumdd}[2]{\displaystyle \sum_{#1}^{#2}}
\newcommand{\sumd}[1]{\displaystyle \sum_{#1}}
\newcommand{\limd}[1]{\displaystyle \lim_{#1}}
\newcommand{\intd}[1]{\displaystyle \int_{#1}}
\newcommand{\intdd}[2]{\displaystyle \int_{#1}^{#2}}
\newcommand{\fracd}[2]{\displaystyle \frac{#1}{#2}}
\newcommand{\bb}[1]{\mathbb{#1}}

\newcommand{\defeq}{\mathrel{\mathpalette{\vcenter{\hbox{$:$}}}=}}
\newcommand{\SU}{\mathrm{SU}}
\newcommand{\SO}{\mathrm{SO}}
\newcommand{\U}{\mathrm{U}}

\newcommand{\PU}{\mathrm{PU}}

\newcommand{\A}{{\mathcal{A}}}
\newcommand{\G}{{\mathcal{G}}}

\newcommand{\C}{{\mathcal{C}}}

\newcommand{\CS}{{\mathcal{CS}}}
\newcommand{\YM}{{\mathcal{YM}}}

\newcommand{\afV}{V}

\newcommand{\afA}{A}

\newcommand{\afa}{a}

\title{The Yang-Mills flow for \\ cylindrical end 4-manifolds}

\author{David L. Duncan}

\begin{document}

\maketitle
\date{}

\begin{abstract}
We establish various existence and uniqueness results for the Yang-Mills flow on cylindrical end 4-manifolds. We also show long-time existence and infinite-time convergence under certain hypotheses on the underlying data.
\end{abstract}

\tableofcontents

	\section{Introduction} \label{Introduction}
	
	The Yang-Mills flow is the flow of a natural vector field on the space of connections on a Riemannian $n$-manifold $Z$. The critical dimension for the flow is $n = 4$, and this is the dimension we consider here. In the closed case (compact with no boundary), this 4-dimensional flow has been studied extensively by many authors \cite{Sed,BL,DonASD,Struwe,KMN,Sch,Schglobal,SSTz,DW,Waldron,Feehan1}. Here we study the case where $Z$ has \emph{cylindrical ends}. Our main results establish short- and long-time existence and uniqueness results under certain hypotheses.

	 Throughout this paper, $Z$ will be a cylindrical end 4-manifold that is oriented and connected. In particular, this means we can write
			$$Z = Z_0 \cup_{Y} \left( \left[0, \infty \right) \times Y \right),$$
			where $Y$ is a closed 3-manifold, and $Z_0$ is a compact manifold with $\partial Z_0 = Y$. We allow the case where $Y$ has multiple connected components. To simplify the exposition, we assume that $Y$ is non-empty, though the results have extensions to the case where $Z$ is closed. We will use the term \emph{(cylindrical) ends} to refer to $\left[0, \infty \right) \times Y$; though at times we will abuse terminology and refer to $Y$ as the `ends' as well. We assume $Z$ is equipped with a \emph{cylindrical end metric} $g$. This means that $g$ restricts on the ends to have the form
			$$g \vert_{\left[0, \infty \right) \times Y} = ds^2 + g^Y,$$
			where $s$ is the coordinate-variable on $\left[0, \infty \right)$, and $g^Y$ is a fixed metric on $Y$.

			Let $G$ be a compact, connected Lie group, and fix a principal $G$-bundle $P \rightarrow Z$. We assume that $P$ restricts on the cylindrical ends to be a product 
			$$P \vert_{\left[0, \infty \right)  \times Y} = \left[0, \infty \right)  \times Q,$$
			for some bundle $Q \rightarrow Y$. Every bundle on $Z$ is equivalent to a bundle of this form.

			To obtain a good analytic problem, we want to consider only those connections on $P$ that have fixed asymptotics down the cylindrical ends of $Z$. For this purpose, fix a flat connection $a$ on $Q$. We assume this is \emph{acyclic}, meaning that $a$ is irreducible and is non-degenerate as a critical point of the Chern-Simons functional; see Section \ref{PerturbationsOnY}. Let
			$$\A(P; a)$$
			denote the set of smooth connections on $P$ that, together with their derivatives, decay rapidly down the cylindrical end to the fixed connection $a$. We will write $\A^{k,p}(P; a)$ for the completion of $\A(P; a)$ relative to the $W^{k,p}$-Sobolev norm, where the derivatives are defined relative to any element of $\A(P;a)$. 
			
			\medskip
			
			Given an initial connection $A_0$ in (a Sobolev completion of) $\A(P; a)$, the \emph{Yang-Mills flow} is given by
			\begin{equation}\label{YMNoPert}
			\partial_\tau A = - d_A^* F_A, \indent A(0) = A_0,
			\end{equation}
			where $A = A(\tau)$ is a path of connections in (a Sobolev completion of) $\A(P;a)$. Our main results pertain to this flow, and are summarized here. See Sections \ref{Short-TimeExistence} and \ref{Long-TimeExistence} for precise statements.

			\medskip
			
			$\bullet$  \emph{Short-time existence and uniqueness:} Assume $a$ is acyclic and 
			$$A_0 \in \A^{1,2}(P; a) \cap \A^{2,p}(P; a)$$ 
			for some $p > 4$. We show that there is some $\tau_0 > 0$ so that a strong solution to (\ref{YMNoPert}) exists, and is unique on the time interval $\left[0, \tau_0 \right)$. As is familiar in the closed case, we obtain a characterization of the maximal existence time in terms of concentration of the curvature. The only difference in the cylindrical end setting is that these concentration points may escape to infinity down the end.  
			
			In our approach to short-time existence, we primarily refer to Feehan's recent monograph \cite{Feehan1} that summarizes and expands upon the original short-time existence proofs by Struwe \cite{Struwe} and Kozono-Maeda-Naito \cite{KMN}. The main point we emphasize below is that, due to the acyclic assumption on $a$, there is no essential analytic difference in passing from the closed case to the cylindrical end case considered here. This effectively comes down to the observation that the Sobolev embedding $W^{1,2}(Z) \hookrightarrow L^4(Z)$ holds even for non-compact 4-manifolds.

			\medskip
			
			 $\bullet $ \emph{Long-time existence:} We give a proof of long-time existence under the following additional assumptions: 
			
			\begin{itemize}
			\item[(i)] There is a natural index associated to the flat connection $a$, and we assume this index is not too high. 
			\item[(ii)] We assume all ASD connections are \emph{ASD-regular}. This term means that the moduli space of ASD connections is cut out transversely, and hence is smooth and of the expected dimension; see Section \ref{PerturbedYang-MillsTheory}.
			\end{itemize}
			Under these assumptions, we prove that the Yang-Mills functional has a positive minimal energy gap $\eta > 0$; see Section \ref{APositiveEnergyGap}. Then we prove that the flow (\ref{YMNoPert}) exists for all time provided the $L^2$-norm of $F_{A_0}^+$ is less than $\eta$. Our proof of this positive energy gap is a bubble-excluding analysis, relying on an understanding of singularities that can form along the flow. Moreover, the above assumptions also exclude bubbling at infinite time.

			\medskip
			
			 $\bullet$ \emph{Infinite-time convergence:} Let $A(\tau)$ be a solution of the flow that exists for all time. Then, under the above assumptions, we show that for $2 \leq q \leq 4$ the $W^{1,q}(Z)$-limit
			$$A_\infty \defeq \lim_{\tau \rightarrow \infty} A(\tau)$$ 
			exists and is an ASD connection in $\A^{1,q}(P; a)$. 
			
			Our proof of convergence at infinite-time relies on several ingredients. First, since we have excluded bubbling, it follows immediately from Uhlenbeck's weak compactness theorem that we have weak subsequential convergence at infinite time to a Yang-Mills connection $A_\infty$, where the convergence is modulo gauge and on compact sets. A priori, this limiting connection may depend on the subsequence chosen, and it may be the case that the asymptotic limits of $A_\infty$ are not $a$ (i.e., $A_\infty$ may belong to $\A^{1,2}(P; a')$ for some other flat connection $a'$). This latter phenomenon is due to the possibility of energy escaping down the cylindrical ends. To exclude these possibilities, we use the ASD-regularity and small energy assumptions again to show that the path $A(\tau)$ is Cauchy in $W^{1,q}$ on the full 4-manifold $Z$. In particular, this implies $A_\infty$ does in fact belong to $\A^{1,q}(P; a)$, as desired. Moreover, the positive energy gap forces $A_\infty$ to be ASD, as opposed to just Yang-Mills.

			\begin{remark}\label{GenREm}
			(a) S\`{a} Earp \cite{SE} considers a similar flow on cylindrical-end K\"{a}hler manifolds. He makes analogous acyclic assumptions on the asymptotic value $a$. 
			
			\medskip
			
			(b) The natural Sobolev constants for the flow are $k = 1$ and $p = 2$, since these are the weakest constants relative to which the Yang-Mills functional is well-defined. However, since the set of $W^{2,2}$-gauge transformations does not form a well-defined group (see Remark \ref{gaugeactionforlargep}), it is difficult to establish any more than a weak solution to the flow when $A_0$ has regularity $W^{1,2}$; see (c), below. Due to this and related reasons, we will typically work with $k, p$ that are sufficiently far above this borderline level.

			\medskip

			The following remarks refer to authors working over \emph{closed} 4-manifolds. 
			
		\medskip
		
		(c) Struwe \cite{Struwe} proved that if $A_0$ is in $W^{1,2}$, then short-time existence and uniqueness holds for (\ref{YMNoPert}) in a weak sense. His results carry over to our setting as well.  See also Remark \ref{StruRemark}.

			\medskip
			
			(d) Schlatter \cite{Schglobal} proved long-time existence under the assumption that $\smash{F^+_{A_0}}$ is $L^2$-small, and the bundle $P$ has small Pontryagin number. Our approach is in many ways similar, with the restriction on the Pontryagin number being replaced by the index assumption on $a$. 
			
			\medskip
			
			(e) Waldron \cite{Waldron} has recently ruled out finite-time bubbling under the assumption that either $F^+$ or $F^-$ does not concentrate in $L^2$. Waldron's arguments are effectively local in nature, and so extend to our case without much trouble. In particular, the long-time existence stated above holds even without the two index and regularity assumptions that we have made. That being said, by including these assumptions, we can exclude bubbling at infinite-time, as well as finite-time. Moreover, with these assumptions, we are able to phrase sufficient conditions for long- and infinite-time existence purely in terms of an energy condition for the initial connection. Indeed, this is the motivation for our approach, since it can be used to study the behavior of the flow under various adiabatic limits of the underlying metric; this is described in more detail \cite{DunAFI}.

			\medskip
			
			(f) Feehan \cite{Feehan1} has obtained similar infinite-time convergence results, where he uses the \L ojasiewicz-Simon's inequality in place of our (rather strong) index and ASD-regularity assumptions. 
			\end{remark}

			In practice, the acyclic and ASD-regularity assumptions mentioned above are frequently \emph{not} satisfied. For example, if $G = SO(3)$ and $Y = S^1 \times \Sigma$, with $\Sigma$ a surface of genus larger than 1, then all flat connections fail to be acyclic. However, the acyclic and ASD-regularity assumptions are \emph{generic} in a certain sense, and so can often be obtained by perturbing the defining equations. Moreover, this perturbation scheme fits in nicely with various standard applications of gauge theory to low-dimensional topology; see \cite{Donfloer}. Consequently, we consider a suitably \emph{perturbed} version of the flow (\ref{YMNoPert}). We discuss the relevant perturbations in Section \ref{Perturbations}. The discussion culminates with Theorem \ref{existencetheorem} which states that, for a large class of cylindrical end 4-manifolds, the acyclic and ASD-regularity assumptions can always be achieved using some such perturbation. The reader who is not interested in this aspect is welcome to skip to Section \ref{ThePerturbedYang-MillsHeatFlow}, and ignore all perturbation terms (denoted by $K$ and ${\bf K}$). Of course, the trade-off is that the discussion may be vacuous if the regularity assumptions are not satisfied.

			\begin{remark}
			See Janner \cite{Janner2} for a similar perturbed Yang-Mills flow over 3-manifolds.
			\end{remark}

	\noindent {\bfseries Acknowledgments:} The author is grateful to his thesis advisor Chris Woodward for his insight and valuable suggestions. He would also like to thank Tom Parker, as well as Paul Feehan and Alex Waldron for their helpful comments and suggestions with an earlier draft. This paper was completed at McMaster University in 2016, and summarizes work completed from 2013 to 2014 while the author was at Michigan State University.

	\section{Gauge theory with perturbations} \label{Perturbations} \label{BasicDefinitions}

In Section \ref{DefinitionsOfThePerturbations} we define a certain class of perturbations that we will use to perturb the flow. After defining this class, we introduce these perturbations into several standard gauge theoretic constructions (e.g., Yang-Mills theory and Uhlenbeck compactness). This is carried out in Sections \ref{PerturbedCSandYMTheory} and \ref{UhlenbeckCompactness}. Section \ref{TheClassOfPerturbations} provides an existence result providing conditions under which the desirable perturbations exist.

	 Before getting into the details of perturbations, we begin by describing our set-up and notation in the absence of a perturbation. Fix a compact, connected Lie group $G$. Since $G$ is compact, its Lie algebra $\frak{g}$ admits an $\mathrm{Ad}$-invariant inner product $\langle \cdot, \cdot \rangle$. In order to appeal to standard index calculations, we choose this inner product as follows. Fix a faithful unitary embedding $G \rightarrow \U(N)$, and use this to pull back the inner product 
\begin{equation}\label{innerproduct}
\langle \xi , \zeta \rangle = \frac{1}{2 \pi^2} \mathrm{tr}(\xi \cdot \zeta^*) = - \frac{1}{2 \pi^2} \mathrm{tr}(\xi \cdot \zeta)
\end{equation}
on $\frak{u}(N) \subset \End(\bb{C}^N)$. The coefficient $(2\pi^2)^{-1}$ is to ensure we obtain integers for certain characteristic numbers appearing below (see Example \ref{ex1} (b) and Lemma \ref{lemma1}).

			Let $R \rightarrow X$ be a principal $G$-bundle over an oriented Riemannian manifold $X$. We will write $R(\frak{g}) \rightarrow X$ for the adjoint bundle associated to $R$, and 
			$$\Omega^k(X, R(\frak{g}))$$ 
			for the space of $k$-forms on $X$ with values in $R(\frak{g})$. We will use similar notation for forms with values in other bundles. The $\mathrm{Ad}$-invariance of the inner product imply that it combines with the wedge to produce a graded-commutative map of the form
			$$\begin{array}{rcl}
			\Omega^k(X, R(\frak{g})) \otimes \Omega^\ell(X, R(\frak{g}))  & \longrightarrow  & \Omega^{k+\ell}(Z, \bb{R}) \\
			 V \otimes W   & \longmapsto & \langle V \wedge W \rangle.
			\end{array}$$
			Similarly, the Lie bracket defines a graded Lie bracket structure on $\Omega^k(X, R(\frak{g}))$, which we denote by $\left[ V \wedge W \right]$.

			We will use 
			$$\A(R), \indent \mathrm{and} \indent \G(R)$$ 
			for the spaces of smooth connections and gauge transformations, respectively, on $R$. Our convention is that $\G(R)$ acts on $\A(R)$ by pullback (this is a right action). The space $\A(R)$ is naturally an affine space modeled on $\Omega^1(X, R(\frak{g}))$, and we use additive notation to indicate the associated action. 
			
			Associated to each connection $A \in \A(R)$ is a covariant derivative 
			$$d_{A} : \Omega^k(X, R(\frak{g})) \longrightarrow \Omega^{k+1}(X, R(\frak{g})).$$ 
			This satisfies
			$$d_{u^*A} W = \mathrm{Ad}(u^{-1}) d_A\left(  \mathrm{Ad}(u) W \right)	, \indent d_{A + V}  = d_A + \left[ V \wedge \cdot \right]$$
			for all $k$-forms $W \in \Omega^k(X, R(\frak{g}))$, 1-forms $V \in \Omega^1(X, R(\frak{g}))$, and all gauge transformations $u \in \G(R)$. We say that a connection $A$ is \emph{irreducible} if the covariant derivative $d_A$ is injective on 0-forms. 
			
			The curvature of a connection is a 2-form $F_{A} \in \Omega^2(X, R(\frak{g}))$. This satisfies
			$$d_A \circ d_A W = \left[ F_A \wedge W \right], \indent F_{u^*A} = \mathrm{Ad}(u^{-1}) F_A, $$
			$$F_{A+V} = F_A + d_A V + \frac{1}{2} \left[ V \wedge V \right]$$
			for all $W \in \Omega^k(X, R(\frak{g}))$, $V \in \Omega^1(X, R(\frak{g}))$, and $u \in \G(R)$.

	\subsection{Definition of the perturbations}\label{DefinitionsOfThePerturbations}
	
		Here we define the relevant class of perturbations. We begin by discussing the asymptotic behavior down the cylindrical end $Y$, then we discuss the perturbation on the rest of $Z$. We refer the reader to \cite{DunAFI} for more details of the various assertions claimed in this section. 	
	
	\subsubsection{Perturbations on $Y$}\label{PerturbationsOnY}
	
Let $Q \rightarrow Y$ be as in the introduction. Fix a map of the form
\begin{equation}\label{pertdef0}
K: \A(Q) \longrightarrow \Omega^2(Y, Q(\frak{g})), \indent a \longmapsto {K}_a.
\end{equation}
We will always assume this is gauge equivariant in the sense that
$$K_{u^* a} = \mathrm{Ad}(u^{-1}) K_a$$
for all $a \in \A(Q)$ and all gauge transformations $u \in \G(Q)$. We will refer to $K$ as a \emph{perturbation on $Y$}. We will use this to perturb the curvature, by setting
$$F_{a, K} \defeq F_a  - K_a.$$
We will say a connection $a \in \A(Q)$ is \emph{$K$-flat} if $F_{a, K} = 0$. 

Denote the linearization of $K$ at $a$ by 
$$dK_a: \Omega^1(Y, Q(\frak{g})) \longrightarrow \Omega^2(Y, Q(\frak{g})).$$ 
We will typically assume $K$ is chosen to satisfy the following.

\smallskip

\begin{customthmQuant}{Axiom 0}\label{axioms0} \emph{
The perturbation $K$ is chosen so that $dK_a$ is symmetric in the sense that}
$$\intd{Y} \langle dK_a(v) \wedge w \rangle  = \intd{Y} \langle v \wedge dK_a(w) \rangle$$
\emph{for all $v, w \in \Omega^1(Y, Q(\frak{g}))$.}
 \end{customthmQuant}
 The next example shows that this axiom is not difficult to arrange. 

\begin{example}\label{exampleexamples}
(a) Fix a function $H: \A(Q) \rightarrow \bb{R}$, and let $(dH)_a \in T^*_a \A(Q)$ be the derivative at $a$. Then define $K_a \in \Omega^2(Y, Q(\frak{g}))$ by
$$(dH)_a v = \intd{Y} \langle K_a \wedge v\rangle$$
for all $v \in \Omega^1(Y, Q\frak{g}))$. (We are using the integral to identify $\Omega^2(Y, Q(\frak{g}))$ with the dual of $T_a \A(Q) = \Omega^1(Y, Q(\frak{g}))$.) Then this satisfies \ref{axioms0}.

\medskip

(b) Here is a variant of the above that will be useful later. Suppose $\Sigma \subset Y$ is an embedded surface that is closed and oriented. Fix a function $h: \A(Q\vert_\Sigma) \rightarrow \bb{R}$, and for $\alpha \in \A(Q\vert_\Sigma)$, define a 1-form $X_\alpha$ by
$$dh_\alpha (\nu)= \intd{\Sigma} \langle X_\alpha \wedge \nu \rangle$$
for all $\nu \in \Omega^1(\Sigma, P(\frak{g}))$. Next, thicken $\Sigma$ up to a neighborhood $U \times \Sigma \subset Y$, for some interval $U$. Fix a function $f: U \rightarrow \bb{R}$ that is supported in the interior of $U$. Then declare
$$Y_a \defeq df \wedge X_{a \vert},$$
where $a \vert$ denotes the restriction of $a$ to $\left\{t \right\} \times \Sigma \subset U \times \Sigma$. This also satisfies \ref{axioms0}.
\end{example}

We will say a $K$-flat connection $a \in \A(Q)$ is \emph{acyclic} if the matrix
\begin{equation}\label{extendedhess}
			 \left(\begin{array}{cc}
			*d_{\afa} - *d{K}_{\afa} & -d_{\afa}\\
			-d_{\afa}^* & 0 
			\end{array}\right)
			\end{equation}	
			is injective as an operator on $\Omega^1(Y, Q(\frak{g})) \oplus \Omega^0(Y, Q(\frak{g}))$; the Hodge star appearing here is the one on $Y$. The primary relevance of \ref{axioms0} is that it implies the matrix (\ref{extendedhess}) is self-adjoint relative to the $L^2$-inner product
			$$(v, w ) \defeq \intd{Y} \langle v \wedge * w \rangle.$$

	\subsubsection{Perturbations on $Z$}
	
Moving to the 4-manifold $Z$, we are interested in gauge equivariant maps of the form
\begin{equation}\label{pertdef}
{\bf K}: \A(P) \longrightarrow \Omega^2(Z, P(\frak{g})), \indent A \longmapsto {\bf K}_A.
\end{equation} 
We will assume that, for each $\ell \geq 1, p \in \left[1, \infty \right]$, the map ${\bf K}$ is smooth relative to the $W^{\ell,p}$-topology on the domain an codomain. We will also assume any such ${\bf K}$ is translationally-invariant on the cylindrical end, in the following sense: Fix $A \in \A(P)$, and write 
$$A \vert_{\left[0, \infty \right) \times Y} = a + p \: ds$$ 
so $a: \left[0, \infty\right) \rightarrow \A(Q)$ is a path of connections and $p : \left[0, \infty \right) \rightarrow \Omega^0(Y, Q(\frak{g}))$ is a path of 0-forms. Then we assume there is some $K$ as in (\ref{pertdef0}) so that
$${\bf K}_A \vert_{\left[0, \infty \right) \times Y} = K_{a}.$$
Any map ${\bf K}$ satisfying the above will be called a \emph{perturbation}, and we will refer to $K$ as the \emph{induced perturbation on $Y$}. We will say that ${\bf K}$ satisfies \ref{axioms0} if the induced perturbation on $Y$ satisfies \ref{axioms0}.

\begin{remark}\label{defofKY}
A particularly special case is when $Z$ is the cylinder $\bb{R} \times Y$. Any perturbation on $Y$ uniquely determines a translationally-invariant perturbation on $\bb{R} \times Y$. We will use ${\bf K}^Y$ to denote perturbations on $Z$ obtained in this way.
\end{remark}

We will want to assume our perturbations satisfy certain uniform bounds. To state the relevant bounds, set
$$\Omega^k \defeq \Omega^k(Z, P(\frak{g})).$$
Then we will use $d^\ell {\bf K}_A : \otimes^\ell \Omega^1 \rightarrow \Omega^2$ to denote the $\ell$th derivative of ${\bf K}$ at $A$. 

\smallskip

\begin{customthmQuant}{Axiom 1}\label{axioms1} \emph{(Analytic axiom)
For any integers $\ell, k \geq 0$, and $p \in \left[1, \infty \right]$, there is a constant $C_{\bf K}(k, \ell, p)$ so that}
\begin{equation}\label{lemma2}
\begin{array}{lcl}
\Vert d^\ell {{\bf K}}_\afA (V_1, V_2, \ldots, V_\ell) \Vert_{W^{k,p}} \\
\indent \indent \leq C_{\bf K}(k, \ell, p) \left(1 + \Vert F_{\afA, {\bf K}} \Vert_{W^{k-1,p}}^{k}  \right) \Vert \afV_1 \Vert_{W^{k,p}} \Vert \afV_2 \Vert_{W^{k,p}} \ldots \Vert \afV_\ell \Vert_{W^{k,p}}
\end{array}
\end{equation}
\emph{for all connections $\afA$, and compactly supported 1-forms $\afV_1, \ldots, \afV_\ell \in \Omega^1(Z, P(\frak{g}))$. All norms are on $Z$.}
 \end{customthmQuant}

For example, when $\ell = k = 0$, this gives a uniform bound of the form
$$\Vert {\bf K}_A \Vert_{L^p} \leq C_{\bf K}(0, 0, p)$$
for all connections $A$. In Section \ref{TheClassOfPerturbations} we will discuss a class of perturbations that satisfy this axiom.

\medskip

As in the 3-dimensional case, we set
$$F_{A, {\bf K}} \defeq F_A - {\bf K}_A.$$
The linearization of the map $A \mapsto F_{A, {\bf K}}$ is the operator
$$d_{A, {\bf K}}  \defeq d_A - d {\bf K}_A: \Omega^1 \longrightarrow  \Omega^2.$$

Just as the covariant derivative $d_A$ is defined on forms of all degrees, we want to extend $d_{A, {\bf K}}$ to an operator on all forms on $Z$. To do this, we define maps

\begin{equation}\label{exactseq}
\Omega^0 \stackrel{d{\bf K}_A}{\longrightarrow} \Omega^1 \stackrel{d{\bf K}_A}{\longrightarrow} \Omega^2 \stackrel{d{\bf K}_A}{\longrightarrow} \Omega^3 \stackrel{d{\bf K}_A}{\longrightarrow}  \Omega^4,
\end{equation}
as follows:

\begin{itemize}
\item Declare $d{\bf K}_A: \Omega^0 \rightarrow \Omega^1$ to be the zero map.
\item The map $d{\bf K}_A: \Omega^1 \rightarrow \Omega^2$ is the linearization of $A \mapsto K_A$, as above.
\item Declare $d{\bf K}_A: \Omega^2 \rightarrow \Omega^3$ to be the Banach space dual to $d{\bf K}_A: \Omega^1 \rightarrow \Omega^2$. Using the identification $(\Omega^i)^* = \Omega^{4-i}$ coming from integration, this can be equivalently defined by the property 

$$\intd{Z} \langle d{\bf K}_A(W) \wedge V \rangle = - \intd{Z} \langle W \wedge d{\bf K}_A(V) \rangle$$
for $W \in \Omega^2, V \in \Omega^1$ (the minus sign is the account for the grading). 
\item Define $d{\bf K}_A: \Omega^3 \rightarrow \Omega^4$ to be the zero map. Note that this is the Banach space dual to $d{\bf K}_A: \Omega^0 \rightarrow \Omega^1$.
\end{itemize}
Of course, the only interesting part of this is the extension to 2-forms. Note that this is consistent with \ref{axioms0} and the requirement that ${\bf K}$ agrees with $K$ on the ends. 

It follows that $d_{A, {\bf K}}: \oplus_k \Omega^k \rightarrow \oplus_k \Omega^k$ is its own Banach space adjoint, up to the usual sign coming from the grading. Similarly, we can form the $L^2$-Hilbert space adjoint by setting
$$d_{\afA, {\bf K}}^* \defeq -(-1)^{(4-k)(k-1)} * d_{\afA, {\bf K}} * : \Omega^k \longrightarrow \Omega^{k-1}.$$ 
This satisfies

$$\left( d_{\afA, {\bf K}} V, W\right) = \left(V, d_{\afA, {\bf K}}^* W \right)$$
for all compactly supported $V \in \Omega^{k-1}, W \in \Omega^{k}$, where now $\left( \cdot, \cdot \right)$ is the $L^2$-inner product on $Z$ coming from the metric.

It will be convenient if $d_{A, {\bf K}}$ and $F_{A, {\bf K}}$ satisfy the Bianchi identity $d_{A, {\bf K}} F_{A, {\bf K}} = 0$, and similar algebraic identities. For this and similar purposes, we impose the following axiom on ${\bf K}$. 

\smallskip

\begin{customthmQuant}{Axiom 2}\label{axioms2} \emph{(Algebraic axiom)
The following holds for each $\afA \in \A(P)$:}

\begin{itemize}
\item[(i)] $d{\bf K}_\afA \circ d{\bf K}_\afA = 0$
\item[(ii)] $d{\bf K}_\afA({\bf K}_\afA) =  0$
\item[(iii)] $\langle {\bf K}_\afA \wedge {\bf K}_\afA \rangle = 0$
\item[(iv)] $d_\afA ({\bf K}_\afA) = - d{\bf K}_\afA(F_\afA).$
\end{itemize}
\end{customthmQuant}

To construct an example of a perturbation satisfying this, repeat the construction of Example \ref{exampleexamples} (b), but interpret $U \times \Sigma$ as a neighborhood in $Z$ (so $U$ is a surface, as opposed to an interval). 

Note that the gauge equivariance of ${\bf K}$ automatically gives
$$d{\bf K}_\afA \left(d_\afA \phi\right) = \left[ {\bf K}_\afA, \phi \right]$$
for all $\phi \in \Omega^0(Z , P(\frak{g}))$. Combining this with \ref{axioms2}, it follows that $d_{A, {\bf K}}$ behaves algebraically like a usual covariant derivative.

\begin{corollary}\label{goodcor}
Assume \ref{axioms2}. Then the following holds for each $\afA \in \A(P)$ and $\phi \in \Omega^0(Z , P(\frak{g}))$:

$$\begin{array}{ll}
\textrm{(Curvature Identity)} & d_{\afA, {\bf K}} \circ d_{\afA, {\bf K}} \phi = \left[ F_\afA, \phi \right]\\
\textrm{(First Bianchi Identity)} & d_{\afA, {\bf K}} F_{\afA, {\bf K}} = 0\\
\textrm{(Second Bianchi Identity)} & d_{\afA, {\bf K}}^* d_{\afA, {\bf K}}^* F_{\afA, {\bf K}} = 0.
\end{array}$$
\end{corollary}

\subsection{Perturbed Chern-Simons and Yang-Mills theory}\label{PerturbedCSandYMTheory}

Let ${\bf K}$ be a perturbation, with $K$ the induced perturbation on $Y$. Unless otherwise specified, we assume these satisfy \ref{axioms0}, \ref{axioms1}, and \ref{axioms2}.

\subsubsection{Chern-Simons theory}\label{PerturbedChern-SimonsTheory}

Define the \emph{perturbed Chern-Simons functional} by setting

$$\CS_{ {K}, P}: \A(Q) \longrightarrow \bb{R} , \indent \CS_{{K}, P}(\afa) \defeq - \frac{1}{2} \intd{Z} \: \langle F_{\afA, {\bf K}} \wedge  F_{\afA, {\bf K}} \rangle,$$
where $\afA$ is any connection in $\A^{1,2}(R; a)$. It follows from the first Bianchi identity that this is independent of the choice of $A$. Similarly, it depends on ${\bf K}$ only through its asymptotic value $K$. 

The perturbed Chern-Simons functional is invariant under the set of gauge transformations on $Q$ that can be homotoped to the identity; for a more general statement, see (\ref{csdeg}). The critical points of $\CS_{K, P}$ are precisely the $K$-flat connections, and the upper left-hand component of the matrix (\ref{extendedhess}) represents the Hessian of $\CS_{K, P}$ relative to the $L^2$-inner product. Consequently, a $K$-flat connection is acyclic if and only if it is (i) irreducible and (ii) a non-degenerate critical point of $\CS_{K,P}$, modulo gauge. 

\medskip

The perturbed Chern-Simons functional can be viewed as a relative characteristic class for 4-manifolds with boundary or cylindrical ends. As such, it is intimately related to an absolute characteristic class for \emph{closed} 4-manifolds. We describe this now. Fix a principal $G$-bundle $R$ over a closed, connected, oriented 4-manifold $X$. Then
$$\kappa(R) \defeq \frac{1}{2} \intd{X} \langle F_A \wedge F_A \rangle$$
depends only on the topological type of $R$. The following examples relate this to standard characteristic classes (recall from (\ref{innerproduct}) that the inner product $\langle \cdot , \cdot \rangle$ is induced from an embedding $G \hookrightarrow \U(N)$).

\begin{example}\label{ex1}
(a) Suppose $G = \SU(N)$ for $N \geq 2$, and the embedding $G \hookrightarrow \U(N)$ from above is the inclusion. Then the Chern-Weil formula gives
$$\kappa(R) =  2 c_2(R) \left[ X \right] \in 2 \bb{Z}.$$

\medskip

(b) Suppose $G = \U(N)$ for $N \geq 2$, and the embedding $G \hookrightarrow \U(N)$ is just the identity. Then 
$$\kappa(R) =  2 \left( c_2(R)  - \frac{1}{2} c_1^2(R) \right) \left[ X \right] \in  \bb{Z}.$$

\medskip

(c) Suppose $G = \SO(r)$ for $r \geq 2$, and the embedding $G \hookrightarrow \SU(N) \subset \U(N)$ is given by the complexified adjoint action of $G$ on $\frak{g}^{\bb{C}}$. Then the induced inner product on $\frak{g}$ is $-(2 \pi^2)^{-1}$ times the Killing form, and 
$$\kappa(R) =  -2(r-2) p_1(R) \left[ X \right] \in 2(r-2) \bb{Z}$$ 
where $p_1$ is the Pontryagin class. Note that this vanishes for $r = 2$, reflecting the fact that $\SO(2)$ is abelian.

\medskip

(d) Suppose $G =\PU(r)$ for $r \geq 2$, and the embedding $G \hookrightarrow \SU(N) \subset \U(N)$ is given by the complexified adjoint action. Then

$$\kappa(R) =  2  q_4(R) \left[ X \right]  \in 2 \bb{Z},$$
where $q_4(R) \in H^4(X, \bb{Z})$ is a $\PU(r)$-generalization of the first Pontryagin number; see \cite{Woody2} and \cite{DunPU}. 
\end{example}

More generally, we have the following. 

\begin{lemma}\label{lemma1}
Fix $G$ and $\langle \cdot, \cdot \rangle$ as above, and let $X$ be a closed, connected, oriented 4-manifold $X$. Then $\kappa(R)$ is an integer for every principal $G$-bundle $R \rightarrow X$. If $G$ is not abelian, then there are principal $G$-bundles $R$ for which $\kappa(R)$ is non-zero.
\end{lemma}

\begin{proof}
The class
$$2c_2 - c_1^2 \in H^4(B \U(N), \bb{Z}) = \bb{Z}\langle c_2 \rangle \oplus \bb{Z} \langle c_1^2 \rangle$$
is a primitive element of the 4th cohomology of the classifying space $B\U(N)$. Moreover, the embedding $G \hookrightarrow \U(N)$ induces an inclusion
$$BG \hookrightarrow B\U(N),$$
and the universal property for $B\U(N)$ shows that this inclusion is unique up to homotopy. Let 
$$\kappa_G \in H^4(G, \bb{Z})$$ 
be the pullback of $2c_2 - c_1^2$ under this embedding. 

Fix a bundle $R \rightarrow X$ and consider its classifying map $\psi_R: X \hookrightarrow BG$. Using this, we can pull back $\kappa_G$ to a class in $H^4(X, \bb{Z})$. By Example \ref{ex1} (b) and the definition of the inner product on $\frak{g}$, this pulled back class is exactly $\kappa(R)  \in \bb{R}$. That is,

$$ \kappa(R) = \left(\psi_R^* \kappa_G\right) \left[ X \right] \in  \bb{Z},$$
which shows $\kappa(R)$ is an integer.

\medskip

To see this is non-zero when $G$ is not abelian, first note that each compact Lie group $G$ has a finite cover that is a product of tori and compact simple Lie groups. In particular, when $G$ is not abelian, there is a Lie group homomorphism
$$\phi: \SU(2) \longrightarrow G$$
with the property that 
$$\phi^* \kappa_G  = j\: \kappa_{\SU(2)} \in H^4(B\SU(2), \bb{Z}) $$
for some non-zero integer $j$ (this integer reflects the aforementioned finite cover of $G$). We already know that there are $\SU(2)$-bundles $R' \rightarrow X$ for which $(\psi_{R'}^*\kappa_{\SU(2)}) \left[ X \right]$ is non-zero. Fixing such a bundle $R'$, define
$$R \defeq R' \times_{\SU(2)} G,$$
where $\SU(2)$ acts on $G$ by the homomorphism $\phi$. Then
$$\kappa(R) = \left(\phi_R^* \kappa_G \right) \left[ X \right] = j \left( \phi_{R'}^* \kappa_{\SU(2)} \right) \left[ X \right] \neq 0,$$
as desired.
\end{proof}

As an application of this characteristic number, let $Q \rightarrow Y$ be a principal $G$-bundle on a 3-manifold, and suppose $u$ is a gauge transformation on $Q$. Then the mapping torus of $u$ is a bundle $Q_u$ over $S^1 \times Y$, and the integer $\kappa(Q_u) \in \bb{Z}$ depends only on $u$ up to isotopy. Moreover, it follows immediately from the definitions that if $a$ is any connection on $Q$ and $K$ is any perturbation, then
\begin{equation}\label{csdeg}
\CS_{K, P}(u^*a) - \CS_{K, P}(a) = \kappa(Q_u) \in \bb{Z}.
\end{equation}

\subsubsection{Yang-Mills theory}\label{PerturbedYang-MillsTheory}

		The \emph{perturbed Yang-Mills functional}, or \emph{energy}, is defined by

$$\YM_{{\bf K}}(\afA) \defeq \frac{1}{2} \Vert F_{\afA, {\bf K}} \Vert^2_{L^2(Z)} =\frac{1}{2} \intd{Z} \: \langle F_{\afA, {\bf K}} \wedge * F_{\afA, {\bf K}} \rangle.$$
We view this as a real-valued function on $\A^{1,2}(P; a)$. This is invariant under the action of the group 
$$\G(P; e)$$ 
of smooth gauge transformations on $P$ that, together with their derivatives, decay rapidly down the cylindrical end to the identity gauge transformation $e$ on $Q$.

\begin{remark}\label{gaugeactionforlargep}
(a) Suppose $a$ is irreducible. Then it follows that the gauge group $\G(P; e)$ acts freely on $\A(P; a)$. See \cite[Prop. 3.7]{DunAFI}. 

\medskip

(b) Let $\smash{\G^{k,p}(P; e)}$ denote the $W^{k,p}$-completion of $\G(P;e)$. When $kp > 4$, this forms a Banach Lie group that acts smoothly and smoothly on $\A^{k,p}(P;a)$. Moreover, the perturbed Yang-Mills functional is invariant under this action. However, the set $\smash{\G^{2,2}(P; e)}$ is not a group, due to the failure of the Sobolev multiplication theorem at the borderline level. 
\end{remark}

		 The critical points of $\YM_{\bf K}$ on $\A^{1,2}(P; a)$ are those connections that satisfy
		$$d_{\afA, {\bf K}}^* F_{\afA, {\bf K}} = 0.$$
		We call these connections \emph{${\bf K}$-YM}. As in the unperturbed case, the perturbed Yang-Mills and Chern-Simons functionals are intimately related. Indeed, for any $A \in \A^{1,2}(P; a)$, we have

\begin{equation}\label{ympertmotheruck}
\YM_{{\bf K}}(\afA)  =   \Vert F_{\afA, {\bf K}}^+ \Vert^2_{L^2(Z)}  +  \CS_{K,P}(\afa),
\end{equation}
where
$$F_{\afA, {\bf K}}^+ \defeq \frac{1}{2} \left(F_{\afA, {\bf K}} + * F_{\afA,{\bf K}} \right)$$ 
is the anti-self dual part. We say a connection is \emph{${\bf K}$-ASD} if $F_{\afA, {\bf K}}^+  = 0$. The ${\bf K}$-ASD connections are automatically $ {\bf K}$-YM by the first Bianchi identity. Moreover, it follows from (\ref{ympertmotheruck}) that if there are any ${\bf K}$-ASD connections in $\A^{1,2}(P; a)$, then they are the \emph{global} minimizers of $\YM_{{\bf K}}$ on $\A^{1,2}(P; \afa)$. 

The linearization of the map $\afA \mapsto F_{\afA, {\bf K}}^+$ at a connection $\afA$ is the operator 
$${d_{\afA, {\bf K}}^+} \defeq \frac{1}{2} (1 + *) d_{A, {\bf K}}: \Omega^1(Z, P(\frak{g})) \rightarrow \Omega^+(Z, P(\frak{g})).$$ 
We will say that a ${\bf K}$-ASD connection $\afA$ is \emph{ASD-regular} if $\smash{d_{\afA, {\bf K}}^+}$ is surjective. We will say that ${\bf K}$ is \emph{ASD-regular} if each of the following holds.

\begin{itemize}
\item The perturbation $K$ satisfies \ref{axioms0}.
\item All $K$-flat connections on $Q$ are acyclic. 
\item For all $K$-flat $a$, every ${\bf K}$-ASD connection $A \in \A^{1, 2}(P; a)$ is ASD-regular. 
\item For each pair $a^-, a^+$ of $K$-flat connections, if $A$ is any ${\bf K}^Y$-ASD connection on $\bb{R} \times Q$ that is asymptotic to $a^\pm$ at $\pm \infty$, then $A$ is ASD-regular. 
\end{itemize}
In the last bullet, ${\bf K}^Y$ is the perturbation on $\bb{R} \times Y$ induced from $K$ as in Remark \ref{defofKY}, and the ${\bf K}^Y$-ASD condition should be defined using the cylindrical metric $ds^2 + g^Y$ on $\bb{R} \times Y$. 

The primary usefulness of ASD-regularity is that it asserts that for each $K$-flat $a$, the moduli space
\begin{equation}\label{modulispace}
\left. \left\{ A  \in \A^{1,p}(P; a) \: \vert \: F^+_{A, {\bf K}} = 0 \right\} \right/ \G^{2,p}(P; e)
\end{equation}
of ${\bf K}$-ASD connections is a smooth manifold. Here we need to assume $p > 2$ in order to have a good gauge group. 

\begin{remark}
Elsewhere in the literature, the term we are calling `ASD-regular' is often simply called `regular'. We have introduced the prefix `ASD' to help distinguish the term from the function-theoretic notion of regularity. 
\end{remark}

\subsubsection{The index}\label{TheIndex}

			Assume \ref{axioms0}. Fix a $K$-flat connection $a$, and assume this is {acyclic}. For $A \in \A^{1,2}(P; a)$, consider the operator
$$d_{A, {\bf K}}^+ \oplus d_{A, {\bf K}}^*: \Omega^1(Z, P(\frak{g})) \longrightarrow \Omega^+(Z, P(\frak{g}))  \oplus \Omega^0(Z, P(\frak{g})).$$
Since $a$ is acyclic, this operator is Fredholm in suitable Sobolev completions of the domain and codomain. In particular, it has a well-defined Fredholm index 
$$\mathrm{Ind}_{K, P}(a)$$
 and this index depends only on $a, K$ (the asymptotic values of $A, {\bf K}$). This index is exactly the dimension of the moduli space (\ref{modulispace}).
 
 Fix a gauge transformation $u$ on $Q$. Then by the argument of \cite[Prop. 3.16]{Donfloer}, we have the following action-index identity
\begin{equation}\label{actionindex}
n_G \left( \CS_{K, P}(u^*{\afa}) - \CS_{K, P} ({\afa})  \right) =  \mathrm{Ind}_{K, P}(u^* {\afa})  - \mathrm{Ind}_{K, P}({\afa}) .
\end{equation}
Here $n_G \geq 0$ is a number depending only on the Lie group $G$, and the choice of inner product on $\frak{g}$. Given our choice of inner product, it follows from (\ref{csdeg}) that $n_G$ is a rational number. When $G$ is not abelian, then $n_G$ is uniquely determined by (\ref{actionindex}) since there are $a, u$ for which both sides are non-zero. Conversely, when $G$ is abelian, both sides of (\ref{actionindex}) are zero for all $a, u$; this reflects the triviality of the group $\pi_3(G) = 0$. In the abelian case, we are therefore free to declare $n_G = 1$.

\begin{example}
(a) Suppose $G = \SU(N)$ and the embedding $G \hookrightarrow \U(N)$ is the identity. Then $n_G = 2r$. See \cite[Prop. 3.16]{Donfloer}.

\medskip

(b) Suppose $G =\PU(r)$ and the embedding $G \hookrightarrow \SU(N) \subset \U(N)$ is given by the complexified adjoint action. Then $n_G = 1$. See \cite[Prop. 6.6]{DunAFI}.
\end{example}

\subsection{Uhlenbeck compactness}\label{UhlenbeckCompactness}

We will need a perturbed version of Uhlenbeck's compactness theorem that keeps track of energy loss down the cylindrical end. To state the version we need, fix a $K$-flat connection $a$. Then we define a \emph{broken trajectory on $Z$ asymptotic to $a$} to consist of the following:
				\begin{itemize}
				\item A tuple $(\afa^0, \afa^1,  \ldots, \afa^J)$ of $K$-flat connections on $Q$ with $\afa^J = a$. 
				\item A connection $\afA^0 \in \A^{1,2}(P; a^0)$ asymptotic to $a^0$.
				\item A tuple $(B^1, \ldots, B^J)$ of connections on $\bb{R} \times Q$.
				\end{itemize}
				These are required to satisfy
				$$\limd{s \rightarrow - \infty} B^j \vert_{\left\{s \right\} \times Y} = \afa^{j-1}		, \hspace{2cm} \limd{s \rightarrow + \infty} B^j \vert_{\left\{s \right\} \times Y} =  \afa^{j}.$$
				We will typically denote a broken trajectory by $(A^0; B^1, \ldots, B^J)$, with the asymptotic $K$-flat connections $a^j$ understood. A broken trajectory is a \emph{broken ${\bf K}$-YM trajectory} (resp. \emph{broken ${\bf K}$-ASD trajectory}) if $A^0$ and the $B^j$ are all ${\bf K}$-YM (resp. ${\bf K}$-ASD). 
				
				Let ${\mathcal{H}}$ be a function space (e.g., $\C^\infty$ or $W^{1,p}$). We will say that a sequence $\afA_n \in \A^{1,p}(P; \afa)$ of connections \emph{converges in ${\mathcal{H}}$} to a broken trajectory if, for each $1 \leq j \leq J$, there is a tuple $\smash{\{s_n^j \small\}_n}$ of positive real numbers satisfying the following:
				\begin{itemize}
				\item For each compact subset $C \subset Z$, the sequence $\afA_n$ converges in ${\mathcal{H}}(C)$ to $\afA^0$.
				\item For each $j$, the sequence $s_n^j$ increases to $\infty$.
				\item For each $n$, we have $s_n^1 <  s_n^2 < \ldots < s_n^J$.
				\item Fix $1\leq j \leq J$, and let 
				$${\tau_{s_n^j}^* \afA_n}$$
				 denote the connection on $\small[-s_n^j, \infty \small) \times Y$ obtained by translating $ \afA_n \vert_{\left[0, \infty \right) \times Y}$. Then for each compact set $ C \subset \bb{R} \times Y$, the sequence $\smash{\tau_{s_n^j}^*  \afA_n}$ converges in $\smash{{\mathcal{H}}(C)}$ to $\smash{B^j}$. 
				\end{itemize}				 		
				We say that the sequence \emph{converges modulo bubbling} if the convergence to $A^0$ holds on the complement of a finite set of points on $Z$, and the convergence to each $B^j$ is on the complement of a finite set of points in $\bb{R} \times Y$ (this set is allowed to depend on $j$). For more details, see \cite[Chapter 6]{MMR} or \cite{Fl1}; see also \cite{S1} for a nice treatment in the closely related case of holomorphic curves.

				Here is a version of Uhlenbeck's compactness theorem with a curvature hypothesis that a priori excludes bubbling.

\begin{theorem}\label{UhlComp1} (Uhlenbeck \cite{U2})
				Assume ${\bf K}$ is a perturbation satisfying \ref{axioms0} and \ref{axioms1}, and so that all $K$-flat connections are acyclic. Fix $p > 2$, and suppose $\afA_n \in \A^{1,p}(P; a)$ is a sequence of smooth $ {\bf K}$-YM connections with
				\begin{equation}\label{2rtrrr}
				\sup_n \YM_{\bf K}(\afA_n) < \infty, \indent \textrm{and} \indent \sup_n \Vert F_{\afA_n} \Vert_{L^\infty(Z)} < \infty.
				\end{equation}
				Then there is a subsequence (still denoted by $A_n$), and a sequence of gauge transformations $U_n \in \A^{2,p}(P; e)$ so that the $U_n^* A_n$ converge in $\C^\infty$ to a broken $ {\bf K}$-YM trajectory asymptotic to $a$. 
				\end{theorem}
				
Due to the presence of the uniform $L^\infty$-bound in (\ref{2rtrrr}), the proof essentially follows from the same analysis as the unperturbed version of Uhlenbeck's theorem \cite{U2}. See \cite[Prop. 3.8]{DunAFI} for details in the presence of a perturbation. The next remark addresses the situation in the presence of bubbling.

				\begin{remark}\label{UhlRemark}
				(a) Suppose $A_n$ is a sequence of connections in $\A^{1,p}(P; a)$ (not necessarily ${\bf K}$-YM) with a uniform bound  on $\Vert F_{A_n} \Vert_{L^2(Z)}$ and $\Vert F_{A_n} \Vert_{L^\infty(Z)}$. Then Uhlenbeck's weak compactness theorem \cite[Theorem B]{Wuc} implies that a subsequence converges weakly in $W^{1,p}$ on compact subsets, after possibly applying suitable gauge transformations. 
				
				More generally, suppose there was a finite set of points $\left\{z_n \right\} \subset Z$ with the property that, for each compact $\smash{B \subset Z \backslash \left\{z_n \right\}}$ there is a uniform bound on $\smash{\Vert F_{A_n} \Vert_{L^\infty(B)}}$. Then modulo gauge, a subsequence converges weakly in $W^{1,p}$ on compact subsets of $Z \backslash \left\{z_n \right\}$. (Note that even if the $A_n$ are ${\bf K}$-ASD, more work needs to be done to conclude that the limit is ${\bf K}$-ASD as well. This is because we have made no assumptions about the behavior of ${\bf K}$ under weak $W^{1,p}$-limits. This is described further in (b).)
				
				\medskip
				
				(b) As described by Kronheimer \cite{Kron}, due to the non-local nature of the typical perturbations appearing in gauge theory (i.e., holonomy perturbations), compactness statements \emph{in the presence of perturbations and bubbling} are rather subtle. For example, suppose the $A_n$ are ${\bf K}$-ASD, but the $\Vert F_{A_n} \Vert_{L^\infty}$ are unbounded. Then one cannot expect to obtain the strong $\C^\infty$-convergence on the complement of the bubbling set, as is the case in the unperturbed setting. However, Kronheimer does prove strong $W^{1,p}$-convergence in the complement of bubbling set, at least for perturbations ${\bf K}$ satisfying the conclusion of \cite[Lemma 10]{Kron}. 
				
				It is perhaps worth emphasizing that the issues here are not so much due to establishing bounds on sequence of connections, but in showing that the perturbation behaves well relative to weakly convergent subsequences. 
				\end{remark}

				As in the unperturbed case, Theorem \ref{UhlComp1} and the acyclic assumption can be used to show that any finite-energy ${\bf K}$-YM connection $A$ on $Z$ is asymptotic to some $K$-flat connection $a \in \A(Q)$; see \cite[Section 4.1]{Donfloer}. Moreover, writing
				$$A\vert_{\left[0, \infty \right) \times Y} = a(s) + p(s) \: ds$$ 
				this convergence is exponential in the sense that $a(s)$ (resp. $p(s)$) converges to $a$ (resp. to $0$) exponentially and in $\C^\infty(Y)$. This implies the following refinement of Uhlenbeck's theorem.

				\begin{corollary}\label{UhlCor}
				Assume ${\bf K}$ is a perturbation satisfying \ref{axioms0} and \ref{axioms1}, and so that all $K$-flat connections are acyclic. Suppose $A_n$ is a sequence of ${\bf K}$-YM connections converging in $\C^\infty$ to a broken ${\bf K}$-YM trajectory $(A^0; B^1, \ldots, B^J)$. Fix $1 \leq p \leq \infty$ and an integer $\ell \geq 0$. Then for each $\delta > 0$, there are compact sets
				$$C_0 \subset Z, \indent  \left[-c_j, c_j \right] \subset \bb{R} \times Y, \indent 1 \leq j \leq J$$
				so that
				$$\Vert A_n - A^0 \Vert_{W^{\ell, p} (C_0)} + \sumdd{j = 1}{J} \Vert \tau_{s_n^j}^*A_n - A^0 \Vert_{W^{\ell, p}(\left[-c_j, c_j \right])}  < \delta,$$
				and 
				$$\begin{array}{l}
				\Vert A_n - a^0 \Vert_{W^{\ell, p}(Z \backslash C_0)} \\
				\indent + \sumdd{j = 1}{J} \Vert \tau_{s_n^j}^*A_n - a^{j-1} \Vert_{W^{\ell, p}(\left(-\infty, -c_j \right]) }  + \sumdd{j = 1}{J} \Vert \tau_{s_n^j}^*A_n - a^{j} \Vert_{W^{\ell, p}(\left[ c_j, \infty \right)) } < \delta,
				\end{array}$$ 
				for all sufficiently large $n$.
				\end{corollary}

				\begin{remark}\label{FloerGluing}				
				Assume ${\bf K}$ is ASD-regular. Then Floer's gluing theorem \cite{Fl1} provides a converse to the above compactness statement for ${\bf K}$-ASD connections. Namely, suppose $(A^0; B^1, \ldots, B^J)$ is a broken ${\bf K}$-ASD trajectory asymptotic to $a$. Then using an implicit function theorem, Floer showed there is a sequence of ${\bf K}$-ASD connections $A_n \in \A^{1,p}(P; a)$ that converge in $\C^\infty$ to $(A^0; B^1, \ldots, B^J)$. 
				\end{remark}

\subsection{Existence of suitable perturbations}\label{TheClassOfPerturbations}
	
	To avoid a vacuous discussion below (in particular, in Section \ref{ConvergenceAtInfiniteTime}), we need some sort of existence statement for perturbations satisfying the axioms above.

	\begin{theorem}\label{existencetheorem}
	Assume $Q \rightarrow Y$ is such that all flat connections are irreducible. Then there exists a perturbation ${\bf K}$ that is ASD-regular and satisfies \ref{axioms0}, \ref{axioms1}, and \ref{axioms2}. 
	\end{theorem}

	\begin{proof}[Sketch of Proof]	Theorem \ref{existencetheorem} was proved in \cite[Prop. 6.11]{DunAFI}. The idea is to restrict to perturbations having the form described in Example \ref{exampleexamples} (b), since these automatically satisfy \ref{axioms0} and \ref{axioms2}. Next, restrict further to the family ${\mathcal{F}}$ of perturbations where the function $h$ is defined by considering holonomy over thickened loops in the surface $\Sigma$; see \cite[Section 5.5]{Donfloer}. Here we are using the notation of Example \ref{exampleexamples} (b). Then each ${\bf K} \in {\mathcal{F}}$ satisfies \ref{axioms1}; see \cite[Prop. 7]{Kron}. The key point, however, is that this family ${\mathcal{F}}$ is large enough to contain a comeager set of perturbations satisfying all conditions of ASD-regularity (except possibly the irreducibility condition for $K$-flat connections). The existence of this comeager set follows from a Sard-Smale argument that is now fairly standard in gauge theory; we refer the reader to Donaldson's book \cite[Section 5.5]{Donfloer} for a nice general treatment. This ultimately comes down to the idea that connections are distinguished by their holonomy. Finally, to obtain irreducibility of $K$-flat connections, use the fact that irreducibility is an open condition, and so the assumptions on $Q$ imply that all $K$-flat connections will be irreducible provided $K$ is sufficiently small. \end{proof}

The next example shows that bundles $Q$ satisfying the hypotheses of Theorem \ref{existencetheorem} are fairly abundant.  		
		
\begin{example}
(a) Suppose $G = \SO(3)$ and $Y$ has positive first Betti number. Fix any non-torsion class $\gamma \in H_1(Y, \bb{Z}_2)$, and define $Q \rightarrow Y$ to be the principal $\SO(3)$-bundle whose Stiefel-Whitney class $w_2(Q) \in H^2(Y, \bb{Z}_2)$ is Poincar\'{e} dual to $\gamma$. Then all flat connections on $Q$ are irreducible. 

This strategy generalizes to $G = \PU(r)$ for $r \geq 2$.

\medskip

(b) Suppose $Y$ is any 3-manifold. Then taking the connect sum with the torus $Y \# T^3$ produces a 3-manifold with positive first Betti number. In particular, for each $r \geq 2$, the manifold $Y \# T^3$ admits a $\PU(r)$-bundle with no reducible flat connections. This strategy is due to Kronheimer-Mrowka \cite{KM}.
\end{example}

	\section{Short-time existence for the flow}\label{ThePerturbedYang-MillsHeatFlow}

Fix a perturbation ${\bf K}$, and let $K$ be the induced perturbation on $Y$. Suppose $a$ is a $K$-flat connection on $Q$, and consider the perturbed Yang-Mills functional 
$$\YM_{ {\bf K}}: \A^{1,2}(P; a) \cap \A^{k,p}(P; a) \longrightarrow \bb{R}.$$ 
We intersect with $\A^{1,2}(P, a)$ to ensure we obtain a finite Yang-Mills value. The $L^2$-gradient of $\YM_{\bf K}$ is the vector field $A \mapsto d_{A, {\bf K}}^* F_{A, {\bf K}}$. The \emph{(perturbed) Yang-Mills flow} is the negative gradient flow of $\YM_{ {\bf K}}$:
\begin{equation}\label{ymheatflow}
\partial_\tau A = - d_{A , {\bf K}}^* F_{A, {\bf K}} , \indent A(0) = A_0,
\end{equation}
where $A_0 \in \A^{1,2}(P; a) \cap \A^{k',p'}(P; a)$ is some fixed initial condition, and the unknown $A$ is a path in $\A^{1,2}(P; a) \cap \A^{k,p}(P; a)$. Here $k', p'$ are Sobolev constants that, for us, will typically be higher than $k, p$ (though ideally $k = k'$ and $p = p'$). The flow (\ref{ymheatflow}) is invariant under the action of the gauge group $\G^{2,2}(P; e) \cap \G^{k+1,p}(P; e)$, provided $(k+1)p > 4$; see Remark \ref{gaugeactionforlargep}. When ${\bf K} = 0$, the flow (\ref{ymheatflow}) is exactly the flow (\ref{YMNoPert}) from the introduction.

In this section we establish short-time existence of the flow (\ref{ymheatflow}) under mild hypotheses on $A_0$. The statements of the results are given in Section \ref{Short-TimeExistence}, with the proofs being deferred to Section \ref{Proofs}. In the intermediary Section \ref{ComparisonWithClosedCase}, we discuss how the cylindrical end case compares with the more standard case where $Z$ is closed.

\subsection{Statements of the short-time existence results}\label{Short-TimeExistence}

Fix a perturbation ${\bf K}$. Throughout this section, we assume this satisfies \ref{axioms0}, \ref{axioms1}, and \ref{axioms2}. We assume further that $a \in \A(Q)$ is a $K$-flat connection that is acyclic. Fix a smooth reference connection 
$$A_{ref} \in \A^{1,2}(P; a);$$ 
define all Sobolev norms relative to this connection (coupled with the Levi-Civita connection on $Z$).

\begin{theorem}  \label{TheoremShort-TimeExistence}(Short-time existence)
Fix $4 <  p < \infty$, as well as an initial condition $A_0 \in \A^{1,2}(P; a) \cap \A^{2, p}(P; a)$. Then there is some $\tau_1 > 0$, and a unique solution 
\begin{equation}\label{weakregularity}
A \in  \C^0\Big(\left[0, \tau_1 \right), \A^{1,2}(P; a)\Big) \cap \C^1\Big(\left(0, \tau_1 \right) \times Z \Big) \cap \C^0_{loc}\Big( \left(0, \tau_1 \right), \A^{2,2}(P; a) \Big)
\end{equation} 
to the perturbed Yang-Mills flow (\ref{ymheatflow}). Moreover, the curvature has regularity at least
\begin{equation}\label{weakFreg}
F_A \in \C^0\Big( \left[0, \tau_1 \right), L^2(Z) \Big) \cap \C^1\Big( \left(0, \tau_1 \right) \times Z\Big) \cap L^\infty_{loc}\Big( \left(0, \tau_1 \right), W^{2,2}(Z) \Big).
\end{equation}
If $A_0$ is $\C^\infty(Z)$, then the solution $A$ is in $\C^\infty\left( \left[0, \tau_1 \right) \times Z \right)$. 
\end{theorem}

We carry the proof out in Section \ref{Proof1}. Note that the acyclic assumption on $a$ implies that all connections in $\A^{1,2}(P; a)$ are irreducible, and so we do not need to impose irreducibility as an additional hypothesis to obtain uniqueness.

\begin{remark}\label{StruRemark}
(a) The flow (\ref{ymheatflow}) is not parabolic due to its invariance under the action of the gauge group. One consequence of this is that the flow is typically not smoothing. For example, suppose the initial condition $A_0$ is a Yang-Mills connection. Then the constant path $A(\tau) = A_0$ clearly solves (\ref{ymheatflow}), but it only has as much regularity as $A_0$. By applying a gauge transformation with low regularity, it is not hard to construct Yang-Mills connections that are $\A^{1,2}(P; a) \cap \A^{k,p}(P; a)$, but not smooth.

\medskip

(b) A more natural initial condition would be to simply assume $A_0$ is in $\A^{1,2}(P; a)$. However, Struwe \cite{Struwe} pointed out that the gauge equivariance of the flow makes it unlikely that one can expect much more than a weak solution to (\ref{ymheatflow}) if $A_0$ only has $W^{1,2}$-regularity, even in the closed, unperturbed case. That being said, Struwe was able to establish a weak solution for $W^{1,2}$-initial conditions, and he showed that the weak solution is gauge equivalent (in a certain sense) to a strong solution. Moreover, his proof extends, without much difficulty, to our situation with cylindrical ends and a perturbation (we discuss this in more detail in Section \ref{ComparisonWithClosedCase}). That is, in our case as well as Struwe's, we have that whenever $A_0 \in \A^{1,2}(P; a)$, there is a unique path
$$A \in \C^0\left( \left[0, \tau_1 \right], \A^{0,2}(P; a) \right) \cap W^{1,2} \left( \left(0, \tau_1 \right) , \A^{0,2}(P; a) \right)$$
satisfying (\ref{ymheatflow}) weakly. This has the additional property that
$$F_{A, {\bf K}} \in \C^0 \left(  \left[0, \tau_1 \right), L^2(Z) \right),$$
and $A$ is gauge equivalent to a smooth solution in the sense described in \cite[Theorem 1.1(i)]{Sch}. 

\medskip

(c) One can improve on the regularity in (\ref{weakregularity}) and (\ref{weakFreg}) in various ways. For example, we will see in the proof that $d_A^* F_A$ is in $L^\infty_{loc}((0, \tau_1) , W^{2,2}(Z))$. 
\end{remark}

As in the closed case \cite{Struwe,Sch}, the maximal existence time for the flow is determined by concentration of energy. The new feature coming from the non-compactness is that it is conceivable the energy concentrates at points that escape down the cylindrical end.

\begin{proposition}\label{EnergyConcentrationTheorem} (Energy concentration)
Under the hypotheses and notation of Theorem \ref{TheoremShort-TimeExistence}, there is some $\eta_{S^4} > 0$ so that the maximal existence time from Theorem \ref{TheoremShort-TimeExistence} is characterized by
$$\overline{\tau} \defeq \sup \left\{ \tau_1 > 0 \: \left| \: \exists R > 0 , \: \sup_{z \in Z, \; 0 \leq \tau \leq \tau_1}   \intd{B_R(z)} \vert F_{A(\tau), {\bf K}} \vert^2  > \eta_{S^4} \right. \right\}.$$
At $\tau = \overline{\tau}$, the curvature concentrates at at most a finite number of points 
$$\begin{array}{r}
(z_{01}, \ldots, z_{0K_0}; (s_{11}, y_{11}), \ldots, (s_{1K_1}, y_{1K_1}); \ldots; (s_{J1}, y_{J1}), \ldots, (s_{JK_J} , y_{JK_J})) \indent \indent \\
\in Z^{K_0} \times (\bb{R} \times Y)^{K_1} \times  \ldots \times (\bb{R} \times Y)^{K_J}
\end{array}$$
in the following sense: 
\begin{itemize}
\item \emph{(Energy concentration on $Z$)} The points $z_{01}, \ldots, z_{0K_0}$ have the property that 
\begin{equation}\label{mustconcentrate1}
\forall 1 \leq k \leq K_0, \: \forall R > 0 , \hspace{.5cm} \limsup_{\tau \nearrow \overline{\tau}  } \intd{B_R(z_{0k})} \vert F_{A(\tau), {\bf K}} \vert^2  > \eta_{S^4};
\end{equation}
\item \emph{(Energy concentration down the cylindrical ends)} There are a finite number 
$$s_1(\tau), \ldots , s_J(\tau)$$ 
of functions $\left[0, \overline{\tau} \right) \rightarrow \bb{R}$ with the property that
$$0 < s_1(\tau) < s_2(\tau) < \ldots < s_J(\tau), \hspace{1cm} \forall 0 \leq \tau < \overline{\tau}$$
and 
$$\lim_{\tau \nearrow \overline{\tau}} s_j(\tau) - s_{j-1}(\tau) = \infty, \hspace{1cm} \forall 1 \leq j \leq J,$$
where $s_0(\tau) \defeq 0$. Moreover, for each $1 \leq j \leq J$, the points
$$(s_{j1}, y_{j1}), \ldots, (s_{jK_j}, y_{j K_j})  \in \bb{R} \times Y$$ 
are such that
\begin{equation}\label{mustconcentrate2}
\forall 1 \leq k \leq K_j, \: \forall R > 0 , \hspace{.5cm}  \limsup_{\tau \nearrow \overline{\tau}  } \intd{B_R((s_j(\tau ) + s_{jk}, y_{jk}))} \vert F_{A(\tau), {\bf K}} \vert^2  > \eta_{S^4}.
\end{equation}
\end{itemize}
\end{proposition}

In the conclusion of Proposition \ref{EnergyConcentrationTheorem}, we are viewing the translated points
$$ (s_j(\tau ) + s_{jk}, y_{jk}) \in \left[0, \infty \right) \times Y \subset Z$$
as belonging to $Z$. Of course, this assumes $\tau$ is close enough to $\overline{\tau}$ so $s_j(\tau) + s_{jk} \geq 0$ for all $j, k$. See Section \ref{ProofOfEnergyConcentrationTheorem} for a proof of this proposition. The proof will show that the quantity $\eta_{S^4}$ can be taken to be the infimum of $\smash{\Vert F_A \Vert^2_{L^2(S^4)}}$ over all non-flat Yang-Mills connections $A$ on bundles over $\smash{S^4}$; we show in Section \ref{APositiveEnergyGap} that this infimum is positive. 

To set up for a more uniform discussion below, we set
$$z^\tau_{0k} \defeq z_{0k}, \indent \mathrm{and} \indent z^\tau_{jk} \defeq ( s_j(\tau) + s_{jk}, y_{jk}), \indent j > 0.$$
We will refer to the points $z_{jk}^\tau$ (for any $j,k$) as the \emph{bubbling points}.

By rescaling around each bubbling point, one can show that a Yang-Mills bubble on $S^4$ forms as $\tau$ approaches the maximal flow time $\overline{\tau}$. 

\begin{proposition}\label{BubbleFormation} (Bubble formation)
At each bubbling point, a non-flat Yang-Mills connection on a bundle over $S^4$ separates, in the following sense: 

In the notation of Proposition \ref{EnergyConcentrationTheorem}, fix $0 \leq j \leq J$ and $1 \leq k \leq K_0$, as well as sequences $\tau_n \nearrow \overline{\tau}$, and $R_n \searrow 0$. Let $d$ be the trivial connection on $\smash{B_{R_1}(z^\tau_{jk})}$ relative to some fixed trivialization of the bundle over $\smash{B_{R_1}(z^\tau_{jk})}$. Write 
 $$A(\tau) \vert_{B_{R_1}(z^\tau_{jk})} = d + M(\tau)$$ 
 for some Lie algebra-valued 1-form $M$. Define a connection on $B_{R_n}(0) \subset \bb{R}^4$ by
 $$A_n(x)  \defeq d + R_n M(\tau_n; z^\tau_{jk} + R_n x )  .$$
Then the $A_n$ converge, modulo gauge and in $W^{1,p}_{loc}(\bb{R}^4)$, to non-flat Yang-Mills connection on $\bb{R}^4$ with finite energy. This Yang-Mills connection extends to a unique smooth non-flat Yang-Mills connection $A_{S^4}$ on some bundle over $S^4$.
\end{proposition}

The proof is given in Section \ref{ProofOfEnergyConcentrationTheorem}. In the statement of the above proposition, we are implicitly assuming that $R_1 > 0$ is small enough so that the ball $\smash{B_{R_1}(z^\tau_{jk}) \subset Z}$ is contractible. Note also that the trivialization of the bundle over $\smash{B_{R_1}(z^\tau_{jk})}$ is independent of $\tau$. This is obvious when $j = 0$ since $z^\tau_{0k} = z_{0k}$. When $j > 0$ this follows because the bundle is translationally-invariant on the cylindrical end.

Even in the presence of bubbling, the connections $A(\tau)$ converge in a rather weak sense on the complement of the bubbling set on $Z$. When bubbles form in finite time, we obtain a statement familiar from the closed setting \cite[Theorem 1.3]{Sch} in the sense that we have $L^2$-convergence on the full 4-manifold.

\begin{proposition}\label{ConvergenceForFiniteTimeBubbling} (Convergence with finite-time bubbling)
Let $A$ be as in the statement of Proposition \ref{BubbleFormation}. Assume the maximal existence time $\overline{\tau} < \infty$ is finite. Then there is a finite-energy connection 
$$A_1 \in \A^{0,2}\Big(P; a\Big) \cap \A^{1,2}_{loc}\Big( P \vert_{Z \backslash \left\{z_{01}, \ldots, z_{0K_0} \right\} }\Big)$$   
with the property that, as $\tau$ increases to $\overline{\tau}$, the connections $A(\tau)$ converge to $A_1$ in $L^2\left( Z \right) \cap W^{1,2}_{loc}\left( Z \backslash \left\{z_{01}, \ldots, z_{0 K_0} \right\} \right)$. Moreover,
\begin{equation}\label{energyinequaltiy}
\YM_{\bf K}(A_1) + \eta_{S^4} \sumdd{j = 0}{J}  \sumdd{k = 0}{K_j} n_{jk}  \leq \liminf_{\tau \nearrow \overline{\tau}} \YM_{\bf K}(A(\tau)),
\end{equation}
for some positive integers $n_{jk}$, where $\eta_{S^4}$ is as in Proposition \ref{EnergyConcentrationTheorem}.
\end{proposition}

We prove Proposition \ref{ConvergenceForFiniteTimeBubbling} in Section \ref{ProofOfConvergenceWithFinite-TimeBubbling}. Note that the convergence of $A(\tau)$ is on the complement of the bubbling points $z_{0k}$ (as opposed to being on the complement of all bubbling points $z_{jk}$ for $j > 0$). This is because, for any fixed compact set $C \subset Z$, the remaining bubbling points $z_{jk}^\tau$, with $j > 0$, all exit $C$ when $\tau$ is sufficiently close to $\overline{\tau}$. On the other hand, the energy inequality (\ref{energyinequaltiy}) remembers all bubbling points. The interpretation of the integers $n_{jk}$ is that the quantity $\eta_{S^4} n_{jk}$ is (a lower bound for) the energy of the Yang-Mills bubble forming at $z^\tau_{jk}$.

\begin{remark}\label{ExtensionToZ}
It follows from Uhlenbeck's theorem on removal of singularities \cite[Theorem 2.1]{U2} that the limiting connection $A_1$ from Proposition \ref{ConvergenceForFiniteTimeBubbling} extends over the bubbling points $z_{0k}$ by possibly modifying the underlying bundle. The finite-energy and $L^2(Z)$-convergence then imply that $A_1$ is gauge equivalent to a connection in $\A^{1,2}(P_1; u_1^* a)$ for some principal $G$-bundle $P_1 \rightarrow Z$ and gauge transformation $u_1$. More precisely, the bundle $P_1 \rightarrow Z$ is such that there is a bundle isomorphism 
$$U_1: P_1 \vert_{Z \backslash \left\{z_{01}, \ldots, z_{0K_0} \right\}} \stackrel{\cong}{\longrightarrow } P \vert_{Z \backslash \left\{z_{01}, \ldots, z_{0K_0} \right\}},$$
and $P_1$ is cylindrical on the end in the sense that
$$P_1 \vert_{\left[s_1, \infty \right) \times Y} \cong \left[s_1, \infty \right) \times Q.$$
 This is the same bundle $Q$ that is associated with $P$, and $s_1 \geq 0$ is large enough so that $\left[s_1, \infty \right) \times Y$ does not contain any of the $z_{0k}$. The pullback connection $U_1^*A_1$ is asymptotic to $u_1^*a$, for some gauge transformation $u_1$ on $Q$. Moreover, this pullback connection extends uniquely to an element of $\A^{1,2}(P_1; u_1^*a)$.

Intuitively, the gauge transformation $u_1$ on $Q$ captures the bubbles that have escaped down the cylindrical end; i.e., those bubbles associated to $\smash{z^\tau_{jk}}$ for $j > 0$. Similarly, the map $U_1$ reflects the bubbles associated to the $z_{0k}$. 
\end{remark}

\subsection{Comparison with the closed case}\label{ComparisonWithClosedCase}

Here we compare our set-up to the case where the base manifold is closed and there are no perturbations. For concreteness, we focus on Struwe's paper \cite{Struwe} as a representative of this latter case. The point we want to emphasize is that Struwe's argument holds \emph{almost verbatim} by simply replacing every $ D = d_{\afA}$ with $d_{\afA, {\bf K}}$, and every $F = F_{\afA}$ with $F_{\afA, {\bf K}}$ ($D$ and $F$ are Struwe's notation). To qualify the term `almost', we begin by addressing the concerns one may have in passing from the closed case to the case with cylindrical ends, as well as how to deal with these concerns. We then move on to address the perturbations.

\medskip

\noindent \emph{Passage to cylindrical end manifolds}

\medskip

In general, when working on cylindrical end manifolds, one needs to watch out for the following.

\begin{itemize}
\item Sobolev embeddings are no longer compact. For example, in dimension 4, there is an embedding $W^{1,2} \hookrightarrow L^2$ in the sense that the one norm bounds the other, however this is not a compact embedding when the base manifold $Z$ is not compact. Fortunately for us, Struwe only uses the compactness of such embeddings for compact intervals of the time variable (e.g., \cite[Proposition 5.2]{Struwe}). 

We note also that there are various strategies for handling compactness-type results on non-compact domains. For example, fix a countable sequence of compact subsets $C_n \subset Z$ that increase and exhaust $Z$. Then one can appeal to compact embeddings on each $C_n$ and then pass to a diagonal subsequence; see \cite[Lemma 4.4.6]{DK}. However, as we just mentioned, such strategies are not necessary here.

\item In the non-compact case the space $L^p$ does not include into $L^q$ when $q < p$ (e.g., $L^\infty$ contains the constant functions which are not in $L^q$ for any $q < \infty$). A related issue is that the Sobolev embedding $W^{k+1,p} \hookrightarrow W^{k, q}$ only holds at the critical level $1 - 4/p =  - 4/q$, but not generally for $1 - 4/p > - 4/q$. (Of course, all of these embeddings do hold on compact manifolds.) Once again, the situation is fortunate for us since Struwe only uses the critical level embeddings; in particular he uses $W^{1,2} \hookrightarrow L^4$, which is fine for us.
\item Fredholm theory on cylindrical end manifolds is a little trickier than for closed manifolds. For example, one would want to know that the elliptic operator 
$$\Delta_{\afA, {\bf K}} = d_{\afA, {\bf K}} d_{\afA, {\bf K}}^* + d_{\afA, {\bf K}}^* d_{\afA, {\bf K}} $$ 
is Fredholm. This is standard in the closed case, but in the presence of cylindrical ends one needs this operator to have `good behavior' along the ends. For us, this `good behavior' condition is satisfied since the $K$-flat connection $\afa$ is assumed to be \emph{acyclic}. See \cite[Section 3]{Donfloer} for a general discussion.

As in the closed case, we also have that $\Delta_{\afA, {\bf K}}$ is self-adjoint; that is, integration by parts holds for $d_{\afA, {\bf K}}$. This is because of \ref{axioms0}, together with the fact that we have restricted to the space $\A^{1,2}(P; \afa)$ so the domain of $d_{\afA, {\bf K}}$ consists of forms on $P$ that vanish at infinity. A nice corollary is that, as in the closed case, the equation $(\partial_\tau + \Delta_{\afA, {\bf K}} ) V = 0$, $V(0) = V_0$ has a unique solution $V$ for each initial condition $V_0 \in W^{1,2}$. This is used to solve for the background connection in \cite[Section 4.1]{Struwe}.
\end{itemize}

Consider, for the moment, the case when ${\bf K} = 0$. Then with the above bulleted comments in mind, Struwe's proof of short-time existence carries over to the cylindrical end setting with no essential change. It therefore remains to discuss how to adapt his proof to accommodate perturbations.

\medskip

\noindent \emph{Handling perturbations}

\medskip

Allowing for non-zero perturbations takes a little more work. There are two potential issues here. One is algebraic, and the other is analytic. To understand the first of these, note that the perturbed objects $d_{A, {\bf K}}$ and $F_{A, {\bf K}}$ are \emph{not} covariant derivatives and curvatures when ${\bf K}$ is non-zero. Nevertheless, by \ref{axioms2}, these behave algebraically like covariant derivatives and curvatures in the sense that they satisfy the conclusions of Corollary \ref{goodcor}. That is, essentially all of the algebra from Struwe's proof carries over to our perturbed setting with only minor changes in notation (here `algebra' means anything with an equal sign). 

Now we discuss how to adapt Struwe's analysis (anything with an inequality). To obtain Struwe's weak solution, at various stages we need a $W^{1,2}$-bound on ${\bf K}_A$. This is furnished by \ref{axioms1}, which implies that the $W^{1,2}$-norm is bounded by the energy of $A$. The energy is decreasing along the flow, so this bound can be taken to be independent of the flow parameter. 

In general, a good rule of thumb for accommodating for perturbations is that any argument that holds for a Riemannian curvature terms (e.g., $\mathrm{Rm} \# W$), will also hold for perturbation terms (e.g., ${\bf K}_\afA \# W$).

\bigskip

To illustrate how all of this can be done, we revisit Struwe's Lemmas 3.1-3.3, which provide the crucial estimates used throughout his (and effectively our) proof of short-time existence and uniqueness. To state the analogous results for us, we assume $a$ is an acyclic $K$-flat connection, and $A$ is a connection in $\A^{1,2}(P; a)$ (or $A$ a path in $\A^{1,2}(P; a)$, depending on context). We set
$$\Delta_{A, {\bf K}} \defeq d_{A, {\bf K}} d_{A, {\bf K}}^* + d_{A, {\bf K}}^* d_{A, {\bf K}},$$
and we use $\nabla_A$ to denote the full covariant derivative associated to $A$ and the Levi-Civita connection. For simplicity, we assume $A$ is smooth.

\begin{lemma} \label{l1} \cite[Lemma 3.1]{Struwe}
There is a constant $C_A$ so that
$$\Vert W \Vert_{W^{2,2}(Z)}^2 \leq C_A \left( \Vert \Delta_{A, {\bf K}} W \Vert^2_{L^2(Z)} + \Vert W \Vert_{L^2(Z)} \right)$$
for all smooth forms $W\in \Omega^i(Z, P(\frak{g}))$ with compact support. The constant $C_A$ depends on $A$ only through the value of $\Vert A - A_{ref} \Vert_{\C^1(Z)}$.
\end{lemma}

\begin{proof}
This is a standard Weitzenb\"{o}ck formula computation. In the perturbed setting, the relevant Weitzenb\"{o}ck formula is

$$\begin{array}{rcl}
\nabla_\afA^* \nabla_\afA W & =& \Delta_{\afA, {\bf K}} W + F_{\afA, {\bf K}} \# W + \mathrm{Rm} \# W\\
&& + d{\bf K}_\afA(d_{\afA, {\bf K}}^* W) + d{\bf K}_\afA^*(d_{\afA, {\bf K}} W)  + d_{\afA, {\bf K}} (d{\bf K}_\afA^* W) + d_{\afA, {\bf K}} ^*(d{\bf K}_\afA W)\\
&&  + {\bf K}_\afA \# W + d{\bf K}_A(d{\bf K}_\afA^* W) + d{\bf K}_\afA^* (d{\bf K}_\afA W),
\end{array}$$
where $W$ is any smooth form on $Z$ with compact support. The verification of this formula is simple: Just expand the perturbation terms on the right and note that they all cancel to yield the usual \emph{unperturbed} Weitzenb\"{o}ck formula. Now Lemma \ref{l1} follows by taking the $L^2$-norm of the above, and estimating the lower order terms (e.g., use \ref{axioms1} for the perturbation terms). Since the Sobolev norms are defined relative to the reference connection $A_{ref}$, one should also use the estimate
$$\Vert W \Vert_{W^{2,2}(Z)} \leq C(A) \left( \Vert W \Vert_{L^2(Z)} + \Vert \nabla^*_A \nabla_A W\Vert_{L^2(Z)} \right),$$
where the constant depends on $A$ through the norm $\Vert A - A_{ref} \Vert_{\C^1(Z)}$. 
\end{proof}

\begin{lemma}\label{l2} \cite[Lemma 3.3]{Struwe}
There are constants $C, \delta > 0$, independent of $A$, with the following significance. Suppose $R > 0$ is such that $A$ satisfies
$$\sup_{x \in Z} \intd{B_R(x)} \: \vert F_A \vert^2 \leq \delta.$$
Then
$$\Vert W \Vert_{L^4(Z)}^2 + \Vert \nabla_A W \Vert_{L^2(Z)}^2 \leq C \left( \Vert d_{A, {\bf K}} W \Vert^2_{L^2(Z)} +\Vert d^*_{A, {\bf K}} W \Vert^2_{L^2(Z)}  + R^{-2} \Vert W \Vert^2_{L^2(Z)} \right)$$
for all smooth forms $W \in \Omega^i(Z, P(\frak{g}))$ with compact support.
\end{lemma}

\begin{proof}
Integrate the above Weitzenb\"{o}ck formula against $W$ to get

$$\begin{array}{rcl}
\Vert \nabla_\afA W \Vert_{L^2(Z)}^2 & = & (\nabla_\afA^* \nabla_\afA W, W)\\
&&\\
& = & \Vert d_{\afA, {\bf K}} W \Vert^2 + \Vert d_{\afA, {\bf K}}^* W \Vert^2 + (F_{\afA, {\bf K}} \# W, W) + (\mathrm{Rm} \# W , W)\\
&& +2 (d_{\afA, {\bf K}}^* W, d{\bf K}_\afA^* W) + 2 (d_{\afA, {\bf K}} W, d{\bf K}_\afA W) \\
&&+ ({\bf K}_\afA \# W, W)+ \Vert d{\bf K}_{\afA} W \Vert^2 + \Vert d{\bf K}_{\afA}^* W \Vert^2 \\
&&\\
& \leq & 2\left( \Vert d_{\afA, {\bf K}} W \Vert^2 + \Vert d_{\afA, {\bf K}}^* W \Vert^2 + \Vert d{\bf K}_{\afA} W \Vert^2 + \Vert d{\bf K}_{\afA}^* W \Vert^2\right)\\
&& + (F_{\afA, {\bf K}} \# W, W) + (\mathrm{Rm} \# W , W)+ ({\bf K}_\afA \# W, W).
\end{array}$$
The metric term $(\mathrm{Rm} \# W , W)$ is bounded since $Z$ has bounded geometry. Similarly, the assumptions on the perturbation ${\bf K}$ provide bounds of the form 
$$ \Vert d{\bf K}_{\afA} W \Vert^2 \leq C \Vert W \Vert^2, \indent \mathrm{and} \indent ({\bf K}_\afA \# W, W) \leq C' \Vert W \Vert^2$$ 
for constants $C, C'$ independent of $\afA$. The condition on $R$ is precisely what is needed to estimate the $F_A$-term using $A$-independent constants. See \cite{Struwe} for more details here.
\end{proof}

Using these lemmas, one can establish various regularity estimates for connections along the flow (\ref{ymheatflow}); e.g., those estimates in \cite[Section 3.2]{Struwe}. For example, write

$$\Vert \left( \partial_\tau + \Delta_{A, {\bf K}}  \right) W \Vert^2_{L^2(L^2)}  = \Vert \partial_\tau W \Vert_{L^2(L^2)}^2 + \Vert W \Vert_{L^2(W^{2,2})}^2 + 2 \left( \partial_\tau W, \Delta_{A, {\bf K}} W \right).$$
Then Lemma \ref{l1} can be used to estimate the cross term, yielding the following.

\begin{lemma}\label{l3} \cite[Lemma 3.2]{Struwe}
There is a constant $C$ (independent of $A$) and a constant $\tau_A > 0$ (depending on $A$ through $\Vert A - A_{ref} \Vert_{\C^1(Z)}$), so that

$$\Vert  \partial_\tau W \Vert_{L^2(L^2)}^2 + \Vert W \Vert_{L^2(W^{2,2})}^2  \leq C\left(  \Vert \left( \partial_\tau + \Delta_{A, {\bf K}}  \right) W \Vert^2_{L^2(L^2)} + \Vert W(0) \Vert_{L^2(Z)} \right)$$
for all smooth maps $W: \left[0, \tau_A \right] \rightarrow \Omega^i(Z, P(\frak{g}))$ with image in the space of compactly supported forms on $Z$. 
\end{lemma}

Here the notation $L^2(L^2)$ is short-hand for the space $L^2(\left[0, \tau_A \right], L^2(Z))$ of $L^2$-maps from the interval $\left[0, \tau_A \right]$ into the Banach space of $L^2$-forms on $Z$. The notation $L^2(W^{2,2})$ is defined similarly.

\subsection{Proofs of short-time existence results}\label{Proofs}

\subsubsection{Proof of Theorem \ref{TheoremShort-TimeExistence}: Uniqueness and short-time existence}\label{Proof1}

In light of the observations of the previous section, and since short-time existence in the closed case is well-treated in the literature (see \cite{DK,Struwe,KMN,Feehan1}), we will only sketch the basic proof of Theorem \ref{TheoremShort-TimeExistence}, emphasizing the aspects that are new to our situation. We refer primarily to the recent monograph \cite{Feehan1} by Feehan, since it is quite exhaustive and provides a nice overview of the various approaches.

\begin{remark}
Though Feehan \cite{Feehan1} has the closed, unperturbed case in mind, he appeals to general results about flows on Banach spaces (also discussed at length in \cite{Feehan1}). These general results apply to our setting for essentially the same reasons Struwe's results do, as discussed in Section \ref{ComparisonWithClosedCase}. Namely, the acyclic assumption on $a$ implies that the relevant function spaces are the Banach spaces $W^{k,p}(Z)$, and the Laplacian $\Delta_{A, {\bf K}}$ is symmetric and Fredholm on these spaces by Lemma \ref{l1}. 
\end{remark}

The (now standard) first step in establishing short-time existence for the Yang-Mills flow is to follow Donaldson's variant of the `de Turck trick'. Here, one first solves the equation

\begin{equation}\label{gaugeequivflow}
\partial_\tau B = - d_{B, {\bf K}}^* F_{B, {\bf K}} - d_{B, {\bf K}} d_{B, {\bf K}}^*(B - A_{ref}), \indent B(0) = A_0,
\end{equation}
for a path $B$ in $\A^{1,2}(P; a)$, where $A_{ref}$ is the fixed smooth reference connection. A solution $B$ to (\ref{gaugeequivflow}) exists on $\left[0, \tau_1 \right)$ for some $\tau_1 > 0$, and this solution has regularity
\begin{equation}\label{regularityofB}
B \in \C^0\left(\left[0, \tau_1 \right), \A^{1,2}(P; a) \cap \A^{2, p}(P; a)\right) \cap  \C^\infty\left(\left(0, \tau_1 \right) \times Z \right).
\end{equation}
Moreover, if $A_0$ is $\C^\infty(Z)$, then it follows from a bootstrapping argument that the flow $B$ is in $\C^\infty\left(\left[0, \tau_1 \right) \times Z \right)$. For a reference, see \cite[Theorems 16.4, 16.5]{Feehan1}.

We note also that since all elements of $\A^{1,2}(P; a) $ are irreducible, it follows from the argument of \cite[Section 6]{Struwe} that the solution $B$ to (\ref{gaugeequivflow}) is unique. See also \cite[Section 19.2]{Feehan1}.

The next step is to transform our solution $B$ from (\ref{gaugeequivflow}) into a solution of the perturbed Yang-Mills flow (\ref{ymheatflow}). To do this, solve the equation
\begin{equation}\label{uequation}
u^{-1} \partial_\tau u = - d_{B, {\bf K}}^* (B - A_{ref} ), \indent u(0) = e
\end{equation}
for a gauge transformation $u$ on $Z$. Given the regularity of $B$, this has a unique solution $u$, with regularity
$$u \in  \C^0\left(\left[0, \tau_1 \right), \G^{1, p}(P; e)\right) \cap  \C^1\left(\left(0, \tau_1 \right),  \G^{1, p}(P; e)\right).$$
Here we are using the assumption that $p > 4$ in order to obtain good Sobolev multiplication results (e.g., a well-defined gauge group). Moreover, if $A_0$ is $\C^\infty(Z)$, then $d_{B, {\bf K}}^* (B - A_{ref} )$ is smooth, and so $u$ is $\C^\infty$ on $\left[0, \tau_1 \right) \times Z$. See \cite[Lemma 19.1]{Feehan1}.

We set $A(\tau) \defeq \left(u(\tau)^{-1} \right)^*B(\tau)$, and so the regularity on $u, B$ give
$$A \in \C^0\left(\left[0, \tau_1 \right), \A^{0, p}(P; a)\right) \cap  \C^1\left(\left(0, \tau_1 \right),  \A^{0, p}(P; a)\right).$$
If $A$ has sufficient regularity, then one can check that it solves (\ref{ymheatflow}), and uniqueness follows as in \cite{Struwe}. In particular, if $A_0$ is smooth, then $A$ is smooth as well. To finish the proof, we need to work under the general hypothesis that $A_0 \in \A^{1,2} \cap \A^{2,p}$, and show that $A$ and $F_A$ have the claimed regularity (\ref{weakregularity}) and (\ref{weakFreg}), respectively.

\begin{remark}\label{Rem2}
(a) Note that the inhomogeneous term in (\ref{uequation}) is smooth for positive time, but not necessarily at time zero since $B(0) = A_0$. Any irregularity of $B$ at time $\tau = 0$ will likely lead to some sort of irregularity of $u(\tau)$ and hence $A(\tau)$, even for positive $\tau$. This is to be expected, given the observation that the flow (\ref{ymheatflow}) is not smoothing; see Remark \ref{StruRemark} (a). 

\medskip

(b) If $A_0$ only has regularity $W^{1,2}$, then it is at this stage that the proof breaks down. This is because one can only show that $u$ is a gauge transformation in $\G^{0, 2}(P; e)$, which is neither a Lie group, nor does it act on the space of connections. Struwe \cite{Struwe} handles this by taking a slightly different approach than sketched here, wherein he replaces the fixed reference connection $A_{ref}$ (which is constant in $\tau$) with a smooth $\tau$-dependent `background' connection. See Remark \ref{StruRemark} (b).
\end{remark}

Gauge equivariance of the curvature and covariant derivative give
$$F_{A, { \bf K}} = \mathrm{Ad}(u) F_{B, {\bf K}} , \indent d_{A, {\bf K}}^* F_{A, {\bf K}} = \mathrm{Ad}(u) d_{B, {\bf K}}^* F_{B, {\bf K}}.$$
The connection $B$ is smooth for positive time, and $u$ is continuous, so this shows that $F_{A, {\bf K}}$ and $d_{A, {\bf K}}^* F_{A, {\bf K}}$ are in $\C^0((0, \tau_0) \times Z)$. (The same conclusion holds in the absence of the perturbation.) We claim that this allows us to interpret
$$\partial_\tau (A - A_{ref}) = - d_{A, {\bf K}}^* F_{A, {\bf K}}$$
as an equation in $\C^0((0, \tau_1) \times Z)$. Indeed, a priori this is only an equation in $L^p$, due to the weak regularity we have on $A$. However, since the right-hand side is continuous, this shows $A \in \C^1(( 0, \tau_1 ), \C^0(Z))$, as claimed. 

Now we continue the bootstrapping. The $\C^1$-regularity on $F_A$ that was claimed in the statement of Theorem \ref{TheoremShort-TimeExistence} follows from the identities
$$\begin{array}{rcl}
\partial_\tau F_A&  =  & \mathrm{Ad}(u^{-1} \partial_\tau u) F_A + \mathrm{Ad}(u) \partial_\tau F_B,\\
&&\\
d^*_{A_{ref}} F_A & = & d^*_A F_A - * \left[ A_{ref} - A \wedge * F_A \right] \\
& = & \mathrm{Ad}(u) d^*_B F_B + * \left[ A - A_{ref} \wedge * \mathrm{Ad}(u)F_B \right], \\
&&\\
d_{A_{ref}} F_A & = &  \left[ A - A_{ref} \wedge \mathrm{Ad}(u) F_A \right] ,
\end{array}$$
since the right-hand side of each is continuous in all variables. Of course, similar computations hold in the presence of a perturbation. 

Moving on to the spatial derivatives of $A$, we write
$$\partial_\tau d_{A_{ref}} (A - A_{ref})  =  \partial_\tau \left( F_A - F_{A_{ref}} - \frac{1}{2} \left[ A - A_{ref} \wedge A - A_{ref} \right] \right).$$
The above observations show that the right-hand side is continuous. Similarly $\partial_\tau d_{A_{ref}}^* (A - A_{ref})$ is continuous. Combining these gives
$$A \in \C^1((0, \tau_1) \times Z).$$

\medskip

Now we will verify that $A$ is continuous at $\tau = 0$ in the $W^{1,2}(Z)$-topology. The flow gives
 \begin{equation}\label{ab}
A(\tau_b) - A(\tau_a) = \intdd{\tau_a}{\tau_b} \: \partial_\tau A \: d \tau = - \intdd{\tau_a}{\tau_b} \: d^*_{A, {\bf K}} F_{A, {\bf K}} \: d \tau.
\end{equation}
Take the $W^{1,2}$-norm of both sides to get
$$\begin{array}{rcl}
\Vert A(\tau_b) - A(\tau_a) \Vert_{W^{1,2}(Z)} & \leq  & \intdd{\tau_a}{\tau_b} \: \Vert d^*_{A, {\bf K}} F_{A, {\bf K}} \Vert_{W^{1,2}(Z)} \: d \tau\\
&&\\
& \leq & C \intdd{\tau_a}{\tau_b}\Big( \Vert d^*_{A, {\bf K}} F_{A, {\bf K}} \Vert_{L^2(Z)} + \Vert d_{A_{ref}}^* d^*_{A, {\bf K}} F_{A, {\bf K}} \Vert_{L^2(Z)}   \Big.\\
&&  \indent \indent \indent \indent \indent \indent \indent\Big. + \Vert d_{A_{ref}} d^*_{A, {\bf K}} F_{A, {\bf K}} \Vert_{L^2(Z)} \Big).
\end{array}$$
We want to show that the right-hand side goes to zero as $\tau_a, \tau_b$ go to zero. For the first two terms, we have
$$\begin{array}{l}
\indent \intdd{\tau_a}{\tau_b} \Big( \Vert d^*_{A, {\bf K}} F_{A, {\bf K}} \Vert_{L^2(Z)} + \Vert d_{A_{ref}}^* d^*_{A, {\bf K}} F_{A, {\bf K}} \Vert_{L^2(Z)} \Big) \\
\leq \vert \tau_a - \tau_b \vert \sup_{\left[ \tau_a, \tau_b \right]} \Big( \Vert d^*_{A, {\bf K}} F_{A, {\bf K}} \Vert_{L^2(Z)} + \Vert d_{A_{ref}}^* d^*_{A, {\bf K}} F_{A, {\bf K}} \Vert_{L^2(Z)} \Big)\\
\leq \vert \tau_a - \tau_b \vert \sup_{\left[ \tau_a, \tau_b \right]}   \Big( \Vert d^*_{B, {\bf K}} F_{B, {\bf K}} \Vert_{L^2(Z)} + \Vert A_{ref} - A \Vert_{\C^0(Z)} \Vert d^*_{B, {\bf K}} F_{B, {\bf K}} \Vert_{L^2(Z)} \Big).
\end{array}$$
In the second line we used the second Bianchi identity. The continuity properties of $B$ at $\tau = 0$ imply that the supremum here is bounded independent of $\tau_a, \tau_b > 0$ (assuming they are far from the maximal time $\tau_1$). In particular, the right-hand side of the above goes to zero as $\tau_a, \tau_b$ go to zero. It remains to show that
$$\limd{\tau_a, \tau_b \searrow 0} \intdd{\tau_a}{\tau_b}  \Vert d_{A_{ref}} d^*_{A, {\bf K}} F_{A, {\bf K}} \Vert_{L^p(Z)}  = 0.$$
For this, differentiate and use the flow equation to get
\begin{equation}\label{integratingthings}
\begin{array}{rcl}
\fracd{d}{d\tau} \frac{1}{2} \Vert d_{\afA, {\bf K}}^* F_{\afA,  {\bf K}} \Vert^2_{L^2(Z)}  & = & - \Vert d_{\afA,  {\bf K}} d_{\afA,  {\bf K}}^* F_{\afA,  {\bf K}} \Vert^2_{L^2(Z)} \\
&& \indent \indent +\left( *\left[ d_{\afA,  {\bf K}}^* F_{\afA,  {\bf K}} \wedge  * F_{\afA,  {\bf K}} \right], d_{\afA,  {\bf K}}^* F_{\afA,  {\bf K}} \right)\\
&& \indent \indent + \left(* d^2_A {\bf K}( d_{\afA,  {\bf K}}^* F_{\afA,  {\bf K}} , *  F_{\afA,  {\bf K}} ) , d_{\afA,  {\bf K}}^* F_{\afA,  {\bf K}} \right)
\end{array}
\end{equation}
Here $d^2_A {\bf K}$ is the second derivative of ${\bf K}$ at $A$. Note that the last two terms on the right are bounded by some constant $C$ that is independent $\tau$, provided $\tau$ is sufficiently small. Integrating (\ref{integratingthings}) over $\left[\tau_a, \tau_b \right]$ then gives
$$\begin{array}{rcl}
\intdd{\tau_a}{\tau_b}  \Vert d_{A_{ref}} d^*_{A, {\bf K}} F_{A, {\bf K}} \Vert_{L^p(Z)} & \leq & \frac{1}{2} \Vert d_{\afA(\tau_a), {\bf K}}^* F_{\afA(\tau_a),  {\bf K}} \Vert^2_{L^2(Z)} \\
&& - \frac{1}{2} \Vert d_{\afA(\tau_b), {\bf K}}^* F_{\afA(\tau_b),  {\bf K}} \Vert^2_{L^2(Z)}  + \vert \tau_b - \tau_a \vert C\\
&&\\
& = &  \frac{1}{2} \Vert d_{B(\tau_a), {\bf K}}^* F_{B(\tau_a),  {\bf K}} \Vert^2_{L^2(Z)} \\
&& - \frac{1}{2} \Vert d_{B(\tau_b), {\bf K}}^* F_{B(\tau_b),  {\bf K}} \Vert^2_{L^2(Z)}  + \vert \tau_b - \tau_a \vert C\\
\end{array}$$
for some constant $C$. The continuity of $B$ at $\tau = 0$ shows that this is going to zero when $\tau_a, \tau_b$ approach $0$. 

\medskip

It remains to establish the $W^{2,2}(Z)$-regularity on $A$ and $F_A$; we begin with the curvature. For this, we have
$$\Vert F_A \Vert_{W^{2,2}(Z)} \leq C \left( \Vert F_A \Vert_{L^2(Z)} + \Vert d_{A_{ref}} d_{A_{ref}}^* F_A \Vert_{L^2(Z)} + \Vert d_{A_{ref}}^* d_{A_{ref}} F_A \Vert_{L^2(Z)}  \right)$$
for a constant $C$ that is independent of $A = A(\tau)$. Now use 
$$d_{A_{ref}} = d_A + \big[ A_{ref} - A \wedge \cdot \big]$$ 
and the Bianchi identity to continue this as
$$\begin{array}{rcl}
\Vert F_A \Vert_{W^{2,2}(Z)} & \leq & C' \Big((1+ \Vert A- A_{ref} \Vert_{\C^1(Z)} ) \Vert F_A \Vert_{L^2(Z)} + \Vert d_A d_A^* F_A \Vert_{L^2(Z)}  \Big.\\
&&\indent \indent \indent \Big. + \Vert A- A_{ref} \Vert_{\C^0(Z)} \Vert F_A \Vert_{W^{1,2}(Z)} \Big).
\end{array}$$
Since $ \Vert d_A d_A^* F_A \Vert_{L^2(Z)}  =  \Vert d_B d_B^* F_B \Vert_{L^2(Z)} $, the right-hand side is in $L^\infty(\left[\tau_a, \tau_b \right])$ for any $0 < \tau_a <\tau_b < \tau_1$. This establishes the claimed regularity for $F_A$ on $(0, \tau_1)$. 

\begin{remark}\label{moreonregularityremark}
(a) By replacing $F_A$ with $d_A^* F_A$, the same argument shows $d_A^* F_A \in L^\infty_{loc}((0, \tau_1), W^{2,2}(Z))$.

\medskip

(b) This argument does not extend to show that $F_A$ is in $L^\infty_{loc}((0, \tau_1), W^{k,2}(Z))$ for $k \geq 3$. This is because the higher derivatives would produce terms involving $ \Vert A- A_{ref} \Vert_{\C^{k-1}(Z)}$, which we can say nothing about unless $A$ has higher regularity. 
\end{remark}

Now to show $A \in \C^0_{loc}(W^{2,2})$, fix a compact set $S \subset (0, \tau_1)$, as well as $\tau_a, \tau_b \in S$. Taking the $W^{2,2}$-norm of (\ref{ab}) gives
$$\begin{array}{rcl}
\Vert A(\tau_b) - A(\tau_a) \Vert_{W^{2,2}(Z)} & \leq  & \intdd{\tau_a}{\tau_b} \: \Vert d^*_{A, {\bf K}} F_{A, {\bf K}} \Vert_{W^{2,2}(Z)} \: d \tau.
\end{array}$$
By Remark \ref{moreonregularityremark} (a), we can bound this by
$$ \Vert A(\tau_b) - A(\tau_a) \Vert_{W^{2,2}(Z)} \leq C \vert \tau_a - \tau_b \vert$$
with a constant $C$ that depends only on the compact set $S$.

\qed

\subsubsection{Proof of Propositions \ref{EnergyConcentrationTheorem} and \ref{BubbleFormation}: Energy concentration and bubble formation}\label{ProofOfEnergyConcentrationTheorem}

Suppose $A(\tau)$ is a solution of (\ref{ymheatflow}) on $\left[0, \tau_1 \right)$ with the regularity of Theorem \ref{TheoremShort-TimeExistence}. Let $\delta > 0$ be as in the statement of Lemma \ref{l2}. It follows from Lemmas \ref{l2} and \ref{l3} that if there is some $R > 0$ with
$$\sup_{z \in Z, \; 0 \leq \tau < \tau_1}   \intd{B_R(z)} \vert F_{A(\tau), {\bf K}} \vert^2  < \delta $$
then $A(\tau)$ can be continuously extended to $\tau = \tau_1$ (hence extended for a positive time past $\tau_1$, as well); see \cite[Lemma 3.6]{Struwe} for a proof. In particular, the quantity $\overline{\tau}$ from the statement of Proposition \ref{EnergyConcentrationTheorem} does indeed characterize the maximal existence time.  

Suppose now that we are in the setting where energy concentration occurs, and define $\overline{\tau} \in \left(0, \infty \right]$ as in Proposition \ref{EnergyConcentrationTheorem}. Now the remaining assertions of Propositions \ref{EnergyConcentrationTheorem} and \ref{BubbleFormation} follow essentially the same rescaling argument given by Schlatter \cite{Sch}. We summarize the details. 

Fix sequences $\tau_n, R_n$ as in the statement of Proposition \ref{BubbleFormation}. Then find $z_n \in Z$ where the quantity 
\begin{equation}\label{znguy}
\intd{B_{R_n}(z_n)} \vert F_{A(\tau_n), {\bf K}} \vert^2 
\end{equation}
is maximized (the curvature decays on the cylindrical ends, so there do indeed exist such $z_n$). Due to the non-compactness of $Z$, there are two cases to consider.

\medskip

\noindent \emph{Case 1: The $z_n$ are contained in a compact subset of $Z$.}

\medskip

\noindent \emph{Case 2: The $z_n$ are not contained in any compact subset of $Z$.}

\medskip

In Case 1, we can pass to a subsequence and assume the $z_n$ converge to some $z_{01} \in Z$. Now define the sequence $A_n$ of connections by rescaling around $z_{01}$ as described in Proposition \ref{BubbleFormation}. Uhlenbeck's compactness theorem implies that the $A_n$ converge, in the sense described in Proposition \ref{BubbleFormation}, to some limiting connection $A_\infty$ on $\bb{R}^4$ with finite energy. The flow equation (\ref{ymheatflow}) rescales in such a way to imply that $A_\infty$ is Yang-Mills. 

\begin{remark}
The rescaling is also such that the perturbation term vanishes in the limit, so $A_\infty$ is Yang-Mills in the usual sense; this uses the assumption from \ref{axioms1} that ${\bf K}_A$ is uniformly bounded in $L^p$. This type of rescaling is carried out explicitly in Section \ref{APositiveEnergyGap}.
\end{remark}
 By removal of singularities, $A_\infty$ extends to a finite-energy Yang-Mills connection on some bundle over $S^4$. Let $\eta_{S^4} > 0$ denote the infimum of all such energy values (this is indeed positive, see Section \ref{APositiveEnergyGap}). The conformal invariance of the energy justifies the appearance of $\eta_{S^4}$ in (\ref{mustconcentrate1}).

Of course, it is possible that there are multiple points $z_n$ where (\ref{znguy}) is maximized. Repeating the above to all such sequences that are contained in some compact subset of $Z$, we obtain a sequence
$$z_{01}, z_{02}, \ldots, z_{0K_0}$$
of bubbling points in $Z$. There are only a finite number $K_0$ of such points because each Yang-Mills bubble has energy at least $\eta_{S^4}$, and the energy along the flow is no greater than the energy of $A_0$. This finishes the analysis for Case 1. 

\medskip

In Case 2, we can pass to a subsequence and assume $z_n \in \left[0, \infty \right) \times Y$ is on the cylindrical end for all $n$. Then we can write $z_n = (s_n, y_n)$ in coordinates, and assume further that the $y_n$ converge to some $y_{j1}$; where $j >0$ is an indexing integer that will be specified at the end. Since the $z_n$ are not contained in any compact subset of $Z$, it follows that
$$\limd{n \rightarrow \infty} s_n = \infty.$$
Let $s_{j}(\tau): \left[0, \infty\right) \rightarrow \left[0, \infty \right)$ be a function with $s_{j}(\tau_n) = s_n$. Then in the language established after Proposition \ref{EnergyConcentrationTheorem}, the associated bubbling point is
 $$z_{j1}^\tau \defeq (s_j(\tau) + s_{j1}, y_{j1}),$$ 
 where $\smash{s_{j1} \defeq 0}$. Just as in Case 1, rescaling as in Proposition \ref{BubbleFormation} produces a non-flat Yang-Mills connection on $S^4$, and this requires energy at least $\eta_{S^4}$. 

Once again, there may be multiple sequences $\left\{z_n \right\}$ realizing Case 2. Suppose $\left\{z_n'\right\}$ is another such sequence. As before, write $z_n' = (s_n', y_n')$ with the $y_n'$ converging to some point $y_{j' k}$ with $j', k$. We determine these indices by comparing this new sequence with the sequence just considered. This comparison comes by considering cases based upon the relative rates at which the $s_n$ and $s_n'$ go to infinity. If they go to infinity at the same rate, then declare $j' \defeq j$ and $k \defeq 2$, and we set
$$s_{j2} \defeq \limd{n\rightarrow \infty} s_n - s_n'.$$
If $s_n$ goes to infinity faster (resp. slower) than $s_n'$, then we want to think of $j$ as being `greater than' (resp. `smaller than') $j'$, and we set $k \defeq 1$. Find a function $s_{j'}(\tau)$ with $s_{j'}(\tau_n)  = s_n'$, and set $s_{j'1} \defeq 0$. Repeating with all such sequences, we arrive at the set of bubbling points described in Proposition \ref{EnergyConcentrationTheorem}. Once again, there can only be a finite number of bubbling points, and hence a finite number $J$ of rates at which these bubbles escape down the cylindrical end. Once all of these bubbling points have been identified, precise values for the indices $j, j'$, etc. can be given according to this `greater than/less than' procedure. \qed

\subsubsection{Proof of Proposition \ref{ConvergenceForFiniteTimeBubbling}: Convergence with finite-time bubbling}\label{ProofOfConvergenceWithFinite-TimeBubbling}

First we will prove $L^2(Z)$-convergence of $A(\tau)$ as $\tau$ increases to $\overline{\tau} < \infty$. We begin with a few simple computations. The flow equation gives 
$$\partial_\tau F_{A, {\bf K}} = - d_{A, {\bf K}} d^*_{A, {\bf K}} F_{A, {\bf K}}.$$ 
Take the $L^2(Z)$-inner product of this with $F_{A, {\bf K}}$ to get
$$\frac{d}{d \tau} \frac{1}{2} \Vert F_{A, {\bf K}} \Vert^2_{L^2(Z)} = - \Vert d^*_{A, {\bf K}} F_{A, {\bf K}} \Vert_{L^2(Z)}^2.$$
Note that $\frac{1}{2}  \Vert F_{A, {\bf K}} \Vert^2_{L^2(Z)} $ and $\Vert F^+_{A, {\bf K}} \Vert^2_{L^2(Z)} $ differ by the constant $\CS_{K, P}(a)$; see (\ref{ympertmotheruck}). In particular, integrating over some interval $\left[\tau_a, \tau_b \right]$ gives
\begin{equation}\label{integratingover}
\intdd{\tau_a}{\tau_b} \Vert d^*_{A, {\bf K}} F_{A, {\bf K}} \Vert_{L^2(Z)}^2 = \frac{1}{2} \Vert F^+_{A(\tau_a), {\bf K}} \Vert^2_{L^2(Z)}  - \frac{1}{2} \Vert F_{A(\tau_a), {\bf K}}^+ \Vert^2_{L^2(Z)} .
\end{equation}

Recall the identity (\ref{ab}). Take the $L^2(Z)$-norm of both sides of (\ref{ab}), and then using H\"{o}lder's inequality in the time-variable to get
\begin{equation}\label{beforeclaim}
\begin{array}{rcl}
\Vert A(\tau_b) - A(\tau_a) \Vert_{L^2(Z)} & \leq & \intdd{\tau_a}{\tau_b} \Vert d_{A, {\bf K}}^* F_{A, { \bf K}}  \Vert_{L^2(Z)} \: d \tau\\
&&\\
& \leq & \vert \tau_b - \tau_a \vert^{1/2} \Vert d_{A, {\bf K}}^* F_{A, {\bf K}} \Vert_{L^2(\left[ \tau_a, \tau_b \right] \times Z)}.
\end{array}
\end{equation}
Combining this with (\ref{integratingover}) gives
$$\Vert A(\tau_b) - A(\tau_a) \Vert_{L^2(Z)}^2 \leq \vert \tau_b - \tau_a \vert \sup_{\left[ \tau_a, \tau_b \right]}  \Vert F^+_{A, {\bf K}} \Vert^2_{L^2(Z)}.$$
The quantity $\Vert F^+_{A, {\bf K}} \Vert^2_{L^2(Z)}$ is non-increasing along the flow, so we have
$$\Vert A(\tau_b) - A(\tau_a) \Vert_{L^2(Z)}^2 \leq \vert \tau_b - \tau_a \vert  \Vert F^+_{A_0, {\bf K}} \Vert^2_{L^2(Z)}.$$
This implies that $A(\tau)$ is $L^2(Z)$-Cauchy as $\tau \nearrow \overline{\tau}$. In particular, the $A(\tau)$ converge in $L^2(Z)$ to some 
$$A_1 \in \A^{0,2}(P; a).$$ 

The $W^{1,2}_{loc}$-convergence to $A_1$ on $Z \backslash \left\{z_{01},\ldots, z_{0K_0} \right\}$ can now be shown using Schlatter's argument for the proof of \cite[Theorem 1.2 (i)]{Sch}, which is local in nature and hence not sensitive to the cylindrical ends. 

Finally, we establish the energy inequality (\ref{energyinequaltiy}). For this, let $\tau_n, R_n$ be as in the statement of Proposition \ref{BubbleFormation}; we may assume the $\tau_n$ are increasing, and the $R_n$ are small. Consider the complement
$$Z_n \defeq Z \backslash \bigcup_{j,k} B_{R_n}\left(z_{jk}^{\tau_n} \right),$$
of the $R_n$-balls around the bubbling points. Then
$$\YM_{\bf K}(A(\tau_n)) = \frac{1}{2} \intd{Z_n} \vert F_{A(\tau_n), {\bf K}} \vert^2 + \sumd{j, k} \frac{1}{2} \intd{B_{R_n}(z_{jk}^{\tau_n} )} \vert F_{A(\tau_n), {\bf K}} \vert^2 .$$
The energy is non-increasing along the flow, so the sequence $\YM_{\bf K}(A(\tau_n)) $ converges to the liminf of $\YM_{\bf K}(A(\tau))$. Hence
$$\limsup_{n \rightarrow \infty} \frac{1}{2} \intd{Z_n} \vert F_{A(\tau_n), {\bf K}} \vert^2 + \eta_{S^4} \sumd{j, k} n_{j,k} \leq \liminf_{\tau \nearrow \overline{\tau}} \YM_{\bf K}(A(\tau)).$$
On the other hand, for each compact set $S \subset Z$, the $W^{1,2}_{loc}$-convergence of the $A(\tau)$ gives
$$\frac{1}{2} \intd{S} \vert F_{A_1, {\bf K}} \vert^2 \leq  \limsup_{n \rightarrow \infty} \frac{1}{2} \intd{Z_n} \vert F_{A(\tau_n), {\bf K}} \vert^2.$$ 
Since this bound is plainly independent of the compact set $S$, we obtain (\ref{energyinequaltiy}). \qed

\section{Long-time existence and convergence at infinite time}\label{ConvergenceAtInfiniteTime}\label{Long-TimeExistence}

Fix a perturbation ${\bf K}$ as well as an acyclic $K$-flat connection $a$ on $Q$. Throughout this section, we assume ${\bf K}$ satisfies the conclusions of Theorem \ref{existencetheorem}. 
				
				In Section \ref{TheIndex}, we associated to $a$ an index $\mathrm{Ind}_{P, K}(a) \in \bb{Z}$, defined as the index of a natural Fredholm operator. Our long-time existence result below relies on a certain restriction of this index. To state this restriction, we consider the extended real number 
$$\mathcal{I}_G \defeq \inf_{R \rightarrow S^4} \:  n_G \: \left|  \kappa(R)  \right|,$$
where the infimum is over all principal $G$-bundles $R \rightarrow S^4$ for which $\kappa(R) \neq 0$. Here $n_G > 0$ is the constant from (\ref{actionindex}) and $\kappa(R)$ is the characteristic number from Section \ref{PerturbedChern-SimonsTheory}. The significance of $\mathcal{I}_G$ is that bubbling cannot occur in any ${\bf K}$-ASD moduli space of dimension smaller than $\mathcal{I}_G$; see the end of Section \ref{PerturbedYang-MillsTheory} for more details on this moduli space. See also \cite[Section 8]{AHS} for a similar discussion for general $G$.

\begin{example}
(a) Suppose $G = \SU(r)$. Then $\mathcal{I}_G = 4r$. 

\medskip

(b) Suppose $G =\PU(r)$. Then $\mathcal{I}_G = 4r$. In the notation of Example \ref{ex1} (d), this can be seen by noting that each $\PU(r)$-bundle $R \rightarrow S^4$ lifts to an $\SU(r)$-bundle $R' \rightarrow S^4$, and $q_1(R) = 2r c_2(R')$.  
\end{example}

In general, when $G$ is abelian, we have $\mathcal{I}_G = \infty$ since $\kappa(R) = 0$ for all $R$. When $G$ is not abelian, it follows from Lemma \ref{lemma1} that $\mathcal{I}_G \in (0, \infty)$ is finite and positive.

\medskip

With these preliminaries, we can state the main result.

\begin{theorem}  \label{heatflowtheorem}
Fix a perturbation ${\bf K}$, and assume this satisfies the conclusion of Theorem \ref{existencetheorem}. Suppose $a \in \A(Q)$ is a $K$-flat connection with the property that
$$\mathrm{Ind}_{K, P}({\afa}) < {\mathcal{I}}_G.$$
Then there is some $\eta(a)  > 0$ (depending on $a$, $Z$, the bundle, the metric, and the perturbation) so the following holds.

 Fix $p > 4$, as well as $A_0 \in \A^{1,2}(P; a) \cap \A^{2, p}(P; a)$, and let $A$ be the solution to the flow (\ref{ymheatflow}) from Theorem \ref{TheoremShort-TimeExistence}. If the initial condition $A_0$ satisfies
$$\Vert F_{\afA_0, {\bf K}}^+ \Vert^2_{L^2(Z)} < \eta(a),$$
then the solution $A$ exists for all time $\tau \in \left[0, \infty \right)$, with the regularity asserted in Theorem \ref{TheoremShort-TimeExistence}. Moreover, for each $2 \leq q \leq 4$ the $A(\tau)$ converge exponentially in $W^{1,q}(Z)$, as $\tau$ approaches $\infty$, to a unique ${\bf K}$-ASD connection $A_\infty \in \A^{1, q}(P; a)$. 
\end{theorem}

In Section \ref{DefinitionOfEta}, we give a fairly concrete definition of the constant $\eta(a)$ from the theorem. The proof of Theorem \ref{heatflowtheorem} occupies Sections \ref{Proof2a} (long-time existence) and \ref{Infinite-TimeConvergence} (infinite-time convergence). Our basic analytic arguments for the proof closely follow those of \cite{Sch,Schglobal,Waldron}. See also Feehan's book \cite{Feehan1} for a thorough treatment of the asymptotics of the flow in the absence of perturbations.

The relevance of convergence in $W^{1,q}(Z)$ for $2 <  q \leq 4$ is that it ensures $A_\infty$ is $\G^{2,q}$-gauge equivalent to a smooth connection (e.g., apply Uhlenbeck compactness to the constant sequence $A_n \defeq A_\infty$).

\subsection{A positive energy gap}\label{APositiveEnergyGap}\label{DefinitionOfEta}

Let $n_G$ be as in (\ref{actionindex}). It will be clear from the proof of Theorem \ref{heatflowtheorem} that the constant $\eta(a)$ appearing in its statement can be taken to be the minimum of the numbers $1, 1/n_G$ and the following three numbers:
\begin{adjustwidth}{15pt}{}
\begin{customthmQuant}{A}\label{A}
The infimum
$$\inf_{A} \Vert F_{\afA}^+ \Vert^2_{L^2(S^4)} .$$
Here the infimum ranges over all connections $A$ on any principal $G$-bundle over $S^4$ that are Yang-Mills, not ASD, and satisfy $\YM(A) \leq \CS_{K, P}(a) +1$. 
 \end{customthmQuant}
 
\medskip 
 
\begin{customthmQuant}{B}\label{B}
The infimum
 $$\inf_{ A} \Vert F_{\afA, {\bf K}}^+ \Vert^2_{L^2(Z)}.$$ 
Here the infimum ranges over all connections $A$ on $P$ that are ${\bf K}$-YM, not ${\bf K}$-ASD, and satisfy $\YM_{{\bf K}}(A) \leq \CS_{K, P}(a) +1$.
\end{customthmQuant}

\medskip 
 
\begin{customthmQuant}{C}\label{C}
The infimum
 $$\inf_{A} \Vert F_{\afA, {\bf K}^Y}^+ \Vert^2_{L^2(\bb{R} \times Y)}.$$ 
Here the infimum ranges over all connections $A$ on $\bb{R} \times Q$ that are ${\bf K}^Y$-YM, not ${\bf K}^Y$-ASD, and satisfy $\YM_{{\bf K}^Y}(A) \leq \CS_{K, P}(a) +1$.
\end{customthmQuant}
\end{adjustwidth}
 For the quantity in \ref{C}, the perturbation ${\bf K}^Y$ is the translationally-invariant perturbation on $\bb{R} \times Q$ from Remark \ref{defofKY}. The metric on $\bb{R} \times Y$ is $ds^2 + g^Y$. 

The next theorem justifies this choice of $\eta(a)$ by stating that each of the infima in \ref{A}-\ref{C} is positive. In light of the identity (\ref{ympertmotheruck}), this positivity can be viewed as saying that the perturbed Yang-Mills functional has a positive energy gap above the minimum energy level given by the anti-self dual connections. (Of course, this minimum energy level is only a theoretical lower bound, since ASD connections may not exist.)

\begin{theorem}\label{epsdef1}
The quantities in \ref{A}, \ref{B}, and \ref{C} are positive.
\end{theorem}

\begin{proof}
The proof we present follows a strategy suggested to us by Chris Woodward. It relies on two simple observations. 

\begin{adjustwidth}{15pt}{}
\begin{customthmQuant}{{Observation 1}}\label{Observation1}
 \emph{Suppose $\afA$ is ${\bf K}$-YM on $Z$. Then the self-dual part $F_{A. {\bf K}}^+$ is $L^2$-orthogonal to the image $\mathrm{im} \: d_{A, {\bf K}}^+ \subset \Omega^+(Z, P(\frak{g}))$. If, in addition, $A$ is \emph{not} ${\bf K}$-ASD, then this self-dual part is non-zero and so the cohomology group}
$$H^+_{A, {\bf K}} \defeq \Omega^+(Z, P(\frak{g})) / \mathrm{im} \: d_{A, {\bf K}}^+$$
\emph{is non-trivial.} 

The same observation holds when $Z$ is interpreted as $S^4$ or $\bb{R} \times Y$, and when ${\bf K}$ is interpreted as zero or ${\bf K}^Y$, respectively. (That is, the situations of \ref{A}, \ref{B}, and \ref{C} are all covered.)
\end{customthmQuant}

\medskip

\begin{customthmQuant}{{Observation 2}}\label{Observation2}
\emph{Suppose $A$ is ${\bf K}$-ASD on $Z$. Then there is some constant $C_A$ so that }
\begin{equation}\label{PI1}
\Vert W \Vert_{W^{1,2}(Z)} \leq C_A \Vert d_{A, {\bf K}}^* W \Vert_{L^2(Z)} = C_A \Vert d_{A, {\bf K}} W \Vert_{L^2(Z)}
\end{equation}
\emph{for all self-dual 2-forms $W \in \Omega^+(Z, P(\frak{g}))$.} 

To see (\ref{PI1}), we note that our assumptions on ${\bf K}$ imply that $A$ is ASD-regular. This means the operator
$$d_{A, {\bf K}}^+ : W^{1,2}(\Omega^1) \longrightarrow L^2(\Omega^+)$$ 
is surjective. Equivalently, this implies the formal adjoint 
$$d_{A, {\bf K}}^* : W^{1,2}(\Omega^+) \longrightarrow L^2(\Omega^1).$$ 
is injective. Then (\ref{PI1}) follows because the formal adjoint $d_{A, {\bf K}}^*$ has closed image (this uses the non-degeneracy of the $K$-flat connection to which $A$ is asymptotic; see \cite[Chapter 3]{Donfloer}). Note that (\ref{PI1}) implies 
$$H^+_{A, {\bf K}} = 0$$ 
for all ${\bf K}$-ASD connections $A$. The same holds with $Z$ and ${\bf K}$ interpreted as $\bb{R} \times Y$ and ${\bf K}^Y$, respectively, so this address the situations in \ref{B} and \ref{C}.

The analogous estimate (\ref{PI1}) for the situation in $\ref{A}$ is a bit different, since we are not working with a perturbation on $S^4$ and cannot use it to ensure the ASD-regular condition holds. Nevertheless, the condition does hold due to the symmetries of $S^4$. Indeed, suppose $A$ is an anti-self dual connection on a bundle over $S^4$. Then the Weitzenb\"{o}ck formula \cite[Equation (6.26)]{FU} reads
$$2 d_A^* d_A W = \nabla_A^* \nabla_A W + 4 W.$$
for self-dual 2-forms $W$, where $\nabla_A$ is the covariant derivative obtained by tensoring the Levi-Civita connection on $S^4$ with $A$. The relevance here is that $\nabla_A$ is injective, and so there is some constant $C$ such that
$$\Vert W \Vert_{L^2(S^4)} \leq C \Vert \nabla_A W \Vert_{L^2(S^4)}$$ 
for all self-dual $W$. Combining this with the Weitzenb\"{o}ck formula, we have
$$\Vert W \Vert_{L^2(S^4)}^2 \leq C^2 \left(2  \Vert d_A W \Vert^2_{L^2(S^4)} - 4 \Vert W \Vert^2_{L^2(S^4)} \right) \leq 2 C^2 \Vert d_A W \Vert^2_{L^2(S^4)} ,$$
which implies $d_A$ is injective. Since $d_A$ has closed image in the relevant Sobolev completions, it follows that 

$$\Vert W \Vert_{W^{1,2}(S^4)} \leq  C_A \Vert d_{A} W \Vert_{L^2(S^4)},$$
which is the $S^4$-version of (\ref{PI1}).
\end{customthmQuant}
\end{adjustwidth}

\begin{remark}\label{uniformC}
Given a fixed energy bound $E$, the constant $C_A$ from (\ref{PI1}) can be chosen to be independent of the ${\bf K}$-ASD connection $A$, provided $\YM_{\bf K} (A) \leq E$. This is because of the following. First, the constant $C_A$ depends only on the gauge equivalence class of $A$. Second, the moduli space of fixed-energy ${\bf K}$-ASD connections has a compactification in terms of broken trajectories and bubbles. Lastly, Floer's gluing theorem shows that if (\ref{PI1}) holds for each connection in a broken trajectory-with-bubbles, then it holds for any connection sufficiently Uhlenbeck-close to that broken trajectory-with-bubbles. 
\end{remark}

We will use the above observations to prove that the quantity in \ref{B} is positive; the quantities in \ref{A} and \ref{C} are essentially special cases, given the observations above. 

\medskip

We argue by contradiction. If the quantity in \ref{B} is zero, then we can find a sequence of $A_n$ of smooth ${\bf K}$-YM connections $A_n$ satisfying $\YM_{\bf K}(A_n) \leq \CS_{K, P}(a) + 1$ and $F_{A_n, {\bf K}}^+ \neq 0$, but with the property that the self-dual curvatures are converging to zero
\begin{equation}\label{ymenergyasdfasdas}
\Vert F_{A_n, {\bf K}}^+ \Vert_{L^2(Z)} \longrightarrow 0.
\end{equation}
Since each of these has finite energy, it follows that there are $K$-flat connections $a_n$ with $A_n \in \A^{1,2}(P; a_n)$. To simplify the discussion, we assume $a_n = a$ for all $n$. The general case reduces to this one by the gauge invariance of the problem, the \emph{uniform} energy bound on the $A_n$, and the fact that there are only finitely many gauge equivalence classes of $K$-flat connections. 

We split our analysis up into cases according to whether bubbles form.

\medskip

\noindent \emph{Case 1: $\sup_n \Vert F_{A_n} \Vert_{L^\infty(Z)} < \infty$ (no bubbles form)}

\medskip

In this case, we appeal to the Uhlenbeck Compactness Theorem \ref{UhlComp1}. Then by passing to a subsequence and composing with suitable gauge transformations, we may assume the $A_n$ converge in $\C^\infty$ to a broken ${\bf K}$-ASD trajectory $(A^0; B^1, \ldots, B^J)$ asymptotic to $a$ (this is ${\bf K}$-ASD by (\ref{ymenergyasdfasdas})). By Floer's gluing theorem \cite{Fl1} (see Remark \ref{FloerGluing}), we can find a second sequence $A_n' \in \A^{1,p}(P; a)$ of ${\bf K}$-ASD connections that converge to the same broken trajectory. Then it follows from Corollary \ref{UhlCor} and the triangle inequality, that 
\begin{equation}\label{linftgoingtoero}
\limd{n \rightarrow 0} \Vert A_n - A_n' \Vert_{L^\infty(Z)} = 0.
\end{equation}

By Observation 1, we have $H^+_{A_n} \neq 0$ for all $n$. This implies we can find self-dual forms $W_n \in \Omega^+(Z, P(\frak{g}))$ with
$$\Vert W_n \Vert_{L^2(Z)} = 1, \indent \mathrm{and} \indent d_{A_n} W_n = 0.$$
Translating to a statement about the nearby ${\bf K}$-ASD connection $A'_n$, the second equality above becomes
$$d_{\afA_n', {\bf K}} W_n = \left[ \afA_n'- \afA_n \wedge W_n \right].$$
By Observation 2, there is some constant $C_{A_n'}$ so that
$$1 \leq \Vert W_n \Vert_{W^{1,2}(Z)} \leq C_{A_n'} \Vert d_{\afA_n' , {\bf K}} W_n \Vert_{L^2(Z)} \leq  C_{A_n'}  \Vert \afA_n'  - A_n \Vert_{L^\infty(Z)}.$$
By Remark \ref{uniformC}, this constant $C_{A_n'} \leq C$ is uniformly bounded and so the right-hand side is going to zero by (\ref{linftgoingtoero}). This is a contradiction, and so we are finished with the analysis of Case 1.

  \medskip

\noindent \emph{Case 2: $\sup_n \Vert F_{A_n} \Vert_{L^\infty(Z)} = \infty$ (bubbles form)}

\medskip

Find points $z_n \in Z$ with
$$c_n^2 \defeq \vert F_{A_n(z_n)} \vert = \Vert F_{A_n} \Vert_{L^\infty(Z)}.$$
After possibly passing to a subsequence, we may assume $c_n \rightarrow \infty$. There are two further cases to consider.

\medskip

\noindent \emph{Subcase 1: There is some compact set $C \subset Z$ with $z_n \in C$ for all $n$.}

\medskip

\noindent \emph{Subcase 2: The $z_n$ are not contained in any compact set in $Z$.}

\medskip

Subcase 2 reduces to Subcase 1 as follows. By passing to a subsequence, we may assume $z_n \in \left[0, \infty \right) \times Y$ for all $n$. Relative to these coordinates, write $z_n = (s_n, y_n)$. By translating by $s_n$, we may view $z_n$ as the point in $\bb{R} \times Y$ with coordinates $(0, y_n)$. This recovers Subcase 1 with $Z = \bb{R} \times Y$ and $C = \left\{0 \right\} \times Y$. It will be clear from the proof of Subcase 1 below, that this translating process does not effect the argument.

We will therefore be done if we can show that Subcase 1 leads to a contradiction. We may pass to a subsequence and assume the $z_n$ converge to some $z_\infty \in Z$. Fix a contractible ball $B(z_\infty) \subset Z$ containing $z_\infty$. For simplicity, we assume the radius is 1. By restricting to $B(z_\infty)$, we may view each $A_n$ as a ${\bf K}$-YM connections on the unit ball $B_1(0) \subset \bb{R}^4$ in Euclidean space, relative to the metric $g$ pulled back from $Z$. Comparing these restrictions to the trivial connection $d$ on $\bb{R}^4$, we can write
$$A_n = d + \sumdd{i = 1}{4} B_{n, i}  \: dx_i$$
for some functions $B_{n, i}: B_1(0) \rightarrow \frak{g}$. For each $n$, define a new connection $A'_n$ by
$$A'_n(x) = d + c_n^{-1} \sumdd{i = 1}{4} B_{n, i}(c_n^{-1} x) \:  dx_i.$$
Then $A'_n$ is a ${\bf K}_n$-YM connection on $B_{c_n}(0) \subset \bb{R}^4$ relative to a metric $g_n$. Here the perturbation is 
$${\bf K}_n \defeq c_n^{-2} {\bf K}$$ 
which is converging uniformly to zero. The metric $g_n$ is given by conformally scaling $g$ about 0; in particular, the $g_n$ are converging to the flat Euclidean metric. 

By conformal invariance, the connection $A_n'$ has energy less than $\YM_{\bf K}(A_n)$, which we have assumed is uniformly bounded. We also have
$$\Vert F_{A_n', {\bf K}_n}^+ \Vert_{L^2(B_{c_n}(0))} \longrightarrow 0.$$
However, the $L^\infty$ norms normalize to give
$$\Vert F_{A_n'} \Vert_{L^\infty(B_{c_n}(0))} = \vert F_{A_n'(0)} \vert = 1.$$
Now we effectively repeat the argument from Case 1 on the non-compact manifold $\bb{R}^4$. Uhlenbeck's compactness theorem together with the $L^\infty$-bound implies we can pass to a subsequence and apply suitable gauge transformations so that the $A_n'$ converge in $\C^\infty$ on compact subsets of $\bb{R}^4$ to some limiting finite-energy ASD connection $A_\infty'$ on $\bb{R}^4$ (there is no perturbation in the defining equation, and the metric is the Euclidean one). We note that this convergence implies that there is some constant $C$ so that, for each compact set $B \subset \bb{R}^4$, we have
\begin{equation}\label{linfbound}
\Vert A_n' - A_\infty' \Vert_{L^\infty(B)} \leq C
\end{equation}
provided $n$ is large enough so $B_{c_n}(0) \subset B$. That this bound is independent of the compact set $B$ chosen follows because the $A_n$ are uniformly bounded on the boundary of the ball $B(z_\infty)$, and so the $A_n'\vert_{\partial B_{c_n}(0)}$ decay uniformly to the trivial connection as $n$ goes to infinity.

To obtain a contradiction, consider the self-dual 2-forms
$$W_n \defeq F_{A_n', {\bf K}_n}^+ / \Vert F_{A_n', {\bf K}_n}^+ \Vert_{L^2(B_{c_n}(0))}.$$
Since $A_n'$ is not ${\bf K}_n$-ASD, this is well-defined. However, since $A_n'$ is ${\bf K}_n$-YM, these satisfy
$$d_{A_n', {\bf K}_n} W_n = 0.$$
The uniform bound in (\ref{linfbound}) combines with the convergence of the ${\bf K}_n$ to zero to imply that the $W_n$ are uniformly bounded in $W^{1,2}$ on compact subsets of $\bb{R}^4$, and this bound is independent of the compact set chosen. Using a bump function, we can extend $W_n$ to a self-dual 2-form on $S^4$ so that
$$\sup_n \Vert W_n \Vert_{W^{1,2}(S^4)} < \infty.$$
 Then the $W_n$ converge strongly in $L^2$ on all of $S^4$ to some limiting self-dual form $W_\infty$. This is non-zero since
$$1 = \Vert W_n \Vert_{L^2(B_{c_n}(0))} \leq \Vert W_n \Vert_{L^2(S^4)}.$$
Finally, since $A_\infty$ is a finite-energy ASD connection on $\bb{R}^4$, it has a unique extension to a finite-energy ASD connection on $S^4$. We plainly have
$$d_{A_\infty} W_\infty = 0,$$
which contradicts Observation 2 since $W_\infty \neq 0$. 
\end{proof}

\subsection{Long-time existence}\label{Proof2a}

In this section and the next, we will repeatedly use fact that the $L^2(Z)$-norms of $F_{A, {\bf K}}$ and $F_{A, {\bf K}}^+$ are non-increasing along the flow. Indeed, the relation (\ref{ympertmotheruck}) and the flow (\ref{ymheatflow}) give
\begin{equation}\label{theflowgibes}
\frac{d}{d\tau} \Vert F_{A(\tau), {\bf K}}^+ \Vert_{L^2(Z)}^2 = \frac{d}{d\tau} \YM_{{\bf K}} (A(\tau)) = - \Vert d_{\afA, {\bf K}}^* F_{\afA, {\bf K}} \Vert_{L^2(Z)}^2.
\end{equation}

Now we turn to establishing the long-time existence assertions of Theorem \ref{heatflowtheorem}. By Theorem \ref{TheoremShort-TimeExistence}, there is some maximal time $\overline{\tau} \in \left(0, \infty \right]$ for which the flow starting at $A_0 \in \A^{1,2}(P; a)$ exists for all $\tau \in \left[0, \overline{\tau}  \right)$. Our goal in this section is to show that, under the hypotheses of Theorem \ref{heatflowtheorem}, we have $\overline{\tau}  = \infty$. 

If $\overline{\tau}_1 \defeq \overline{\tau}$ is finite, then it follows from Proposition \ref{ConvergenceForFiniteTimeBubbling} and Remark \ref{ExtensionToZ}, that there is some bundle $P_1 \rightarrow Z$, a gauge transformation $u_1$ on $Q$, and a connection 
$$A_1 \in \A^{1,2}(P_1, u_1^*a)$$ 
so that the $A(\tau)$ converge to a pullback of $A_1$, and 
$$\left\{ \textrm{energy of bubbles} \right\}  + \YM_{\bf{K}}(A_1) \leq \YM_{\bf{K}}(A_0).$$
In fact, we can say a little more: We have assumed that $\smash{\Vert F_{A_0, {\bf K}}^+ \Vert_{L^2(Z)}^2}$ is no greater than the constant $\eta(a)$. This $L^2$-norm is non-increasing along the flow, and is conformally invariant. In particular, it follows from \ref{A} in the definition of $\eta(a)$ that each Yang-Mills bubble is in fact ASD. Energy quantization for ASD connections on $S^4$ implies that each has energy at least ${\mathcal{I}}_G/ n_G$ (these constants were discussed before the statement of Theorem \ref{heatflowtheorem}). In particular, the assumption that at least one bubble forms implies

\begin{equation}\label{smoothconneciton}
{\mathcal{I}}_G/ n_G + \YM_{\bf{K}}(A_1)  \leq \YM_{{\bf K}}(A_0).
\end{equation}

Now we want to start the flow over again, but with initial condition $A_1$ in place of $A_0$. However, from what we have at this point, it may not be the case that $A_1$ has enough regularity to apply the Short-Time Existence Theorem \ref{TheoremShort-TimeExistence}. Nevertheless, for the argument that follows, it suffices to replace $A_1$ by any smooth connection $A_1'$ in $\A^{1,2}(P_1, u_1^*a)$, provided the energy of $A_1'$ is no greater than the energy of $A_1$. Such connections always exist. For example, take $A_1'$ to be any connection along Donaldson's flow (\ref{gaugeequivflow}) with initial condition $A_1$. This flow is smoothing, and the second Bianchi identity shows that it is energy non-increasing.

In summary, we may therefore assume that $A_1$ is a \emph{smooth} connection in $\A^{1,2}(P_1, u_1^*a)$ satisfying (\ref{smoothconneciton}). Now repeat the above procedure with $A_1$ in place of $A_0$. Continuing inductively, there are a number of times $\overline{\tau}_1, \ldots, \overline{\tau}_L$ at which bubbles can form. Associated to each $\overline{\tau}_\ell$ is a gauge transformation $u_\ell$, a bundle $P_\ell$, and a smooth connection $A_\ell \in \A^{1,2}(P_\ell, u_\ell^*a)$ satisfying
$$\ell {\mathcal{I}}_G/ n_G  + \YM_{\bf{K}}(A_k)   \leq \YM_{\bf{K}}(A_0).$$
This shows that there can be only finitely many such times $L \geq 1$ at which bubbles form. Then the flow starting at $A_L$ exists for all time. We will denote this flow by $A(\tau)$, with the understanding that it is valid for $\tau \geq \overline{\tau}_L$.

Now we wish to take the infinite-time limit. Towards this end, we make the following claim.

\begin{adjustwidth}{0pt}{}
\begin{customthmQuant}{Claim 1}\label{Claim1Long-Time} 
\emph{There is a gauge transformation $u_{L+1}$ on $Q$, and a broken ${\bf K}$-ASD trajectory $(A_{L+1}; B_1, \ldots, B_J )$ that is asymptotic to $u^*_{L+1} a$, and satisfies}
$$ \YM_{{\bf K}}(A_{L+1}) + \sumdd{j = 1}{J} \YM_{{\bf K}_Y} (B_j) \leq \YM_{\bf{K}}(A_L).$$
\end{customthmQuant}
\end{adjustwidth}

\medskip

(Note that the analysis from Proposition \ref{ConvergenceForFiniteTimeBubbling} is no longer valid in the infinite-time regime.) To prove the claim, we fix a sequence $\tau_n \rightarrow \infty$, and appeal to Uhlenbeck's weak compactness theorem applied to the sequence $A(\tau_n)$; see Remark \ref{UhlRemark} (a). This sequence converges weakly in $\smash{W^{1,p}_{loc}(Z \backslash B)}$, modulo gauge and on the complement of some finite bubbling set $B \subset Z$, to a connection 
$$A_{L + 1} \in \A^{1,p}_{loc}\left(P \vert_{Z \backslash B} \right).$$ 
By standard infinite-time analysis for flows, it follows that $A_{L+1}$ is ${\bf K}$-YM on the complement of the bubbling set $B$. It also satisfies the energy bound
$$\YM_{{\bf K}}(A_{L+1}) \leq \YM_{{\bf K}} (A_0) =  \CS_{K, P}(a) + \Vert F_{A_0}^+ \Vert^2_{L^2(Z)} < \CS_{K, P}(a) + \eta(a).$$
This implies two things. First, $A_{L+1}$ has finite energy, and so extends over $B$ by removal of singularities. Second, we have defined $\eta(a)$ so that $\eta(a) \leq 1$; hence $\YM_{{\bf K}}(A_{L+1})  < \CS_{K, P}(a) + 1$. The relevance of this latter estimate becomes clear when coupled with the bound
$$\Vert F^+_{A_{L+1}, {\bf K}} \Vert^2_{L^2(Z)} \leq \liminf_{\tau \rightarrow \infty} \Vert F^+_{A(\tau), {\bf K}} \Vert^2_{L^2(Z)}  \leq  \Vert F^+_{A_{0}, {\bf K}} \Vert^2_{L^2(Z)}  < \eta(a).$$
For then, it follows from \ref{B} in the definition of $\eta(a)$ that $A_{L+1}$ must actually be ${\bf K}$-ASD. 

Since $A_{L+1}$ has finite-energy, it is asymptotic to some $K$-flat connection $a_1$. If $a_1$ is gauge equivalent to $a$, then we are done with the proof of the claim. Otherwise, by performing suitable translations on the cylindrical end (see Section \ref{UhlenbeckCompactness}, and the references therein), and repeating the above argument, we can complete $A_{L+1}$ to a broken ${\bf K}$-YM trajectory $(A_{L+1}; B_1, \ldots, B_J )$, with $B_J$ asymptotic, modulo gauge, to $a$. That each $B_j$ is actually ${\bf K}_Y$-ASD follows from the same argument we used to show $A_{L+1}$ is ${\bf K}$-ASD; this time one should use \ref{C} in place of \ref{B}. This finishes the proof of \ref{Claim1Long-Time}.

\medskip

With \ref{Claim1Long-Time} in hand, we now have
\begin{equation}\label{combinewith}
\begin{array}{rcl}
{\mathcal{I}}_G / n_G + \YM_{\bf K}(A_{L+1}) + \sumdd{j = 1}{J} \YM_{{\bf K}^Y}(B_j) & \leq & \YM_{\bf K}(A_0) \\
& < &  \CS_{K, P}(a) + \eta(a) \\
&&\\
&\leq & \CS_{K, P}(a) +1/n_G.
\end{array}
\end{equation}
In the last line we used another defining condition on $\eta(a)$ from Section \ref{APositiveEnergyGap}. We will now focus on the energies appearing on the left. 

Since these connections are anti-self dual, these energies are topological. To compute these energies, let $a_1$ be the asymptotic limit of $A_{L+1}$, and $a_j$ the asymptotic limit of $B_j$ at $+\infty$; hence $a_{J} = u^*_{L+1} a$. Then since $A_{L+1}$ is ${\bf K}$-ASD, we have
$$\YM_{{\bf K}}(A_{L+1}) = \CS_{K, P}(a_0).$$
The version of this for $B_j$ is
$$\YM_{{\bf K}^Y} (B_j) = \CS_K(a_{j-1} , a_j),$$
where $\CS_{K}(b^-, b^+)$ is the perturbed Chern-Simons functional for $\bb{R} \times Q$ with asymptotics at $\pm \infty$ given by $b^\pm$. These Chern-Simons functionals are defined by integrals and so are additive in their arguments 
$$\CS_{K, P}(a) + \CS_K(a, b) = \CS_{K, P}(b), \indent \CS_K(a, b) + \CS_{K}(b, c) = \CS_K(a, c).$$
This gives
$$\begin{array}{rcl}
  \YM_{\bf K}(A_{L+1}) +  \sum_{j=1}^{J} \YM_{{\bf K}^Y}(B_j) & = &   \CS_{K, P}(a_0 ) \\
 && + \smash{ \sum_{j = 1}^{J} \CS_{K}(a_{j-1}, a_j)}\\
 &&\\
& = & \CS_{K, P}(u_{L+1}^*a)\\
&&\\
& =& \CS_{K, P}(a)  + \mathrm{Ind}_{K, P}(u_{L+1}^*a) / n_G\\
&& \indent  - \mathrm{Ind}_{K, P}(a)/n_G,
\end{array}$$
where we used (\ref{actionindex}) in the last line. We have assume $\mathrm{Ind}_{K, P}(a) < {\mathcal{I}}_G$. Since $\mathrm{Ind}_{K, P}(a)$ and $ {\mathcal{I}}_G$ are integers, their difference is at least 1, and so

$$\begin{array}{l}
 \YM_{\bf K}(A_{L+1}) +  \sumdd{j=1}{J} \YM_{{\bf K}^Y}(B_j) \\
 \indent \indent  \geq   \mathrm{Ind}_{K, P}(u_{L+1}^*a) / n_G  +  \CS_{K, P}(a) + 1 / n_G - {\mathcal{I}}_G / n_G .
 \end{array}$$
 Combining this with (\ref{combinewith}) gives 
\begin{equation}\label{comsasfa}
\mathrm{Ind}_{K, P}(u_{L+1}^*a) / n_G  +  \CS_{K, P}(a) + 1 / n_G < \CS_{K, P}(a) +1/n_G.
 \end{equation}
 Hence $\mathrm{Ind}_{K, P}(u_{L+1}^*a)  \leq -1$. Our desired contradiction will now follow from the next claim.

\begin{adjustwidth}{0pt}{}
\begin{customthmQuant}{Claim 2}\label{Claim2Long-Time} 
\emph{The integer $\mathrm{Ind}_{K, P}(u_{L+1}^*a)$ is non-negative.}
\end{customthmQuant}
\end{adjustwidth}

\medskip 

To see this, recall that the quantity $\mathrm{Ind}_{K, P}(u_{L+1}^*a) $ is the expected dimension of the moduli space ${\mathcal{M}}_{ASD}(u_{L+1}^*a; {\bf K})$ of ${\bf K}$-ASD connections that are asymptotic to $u_{L+1}^*a$. We have assumed that ${\bf K}$ is ASD-regular, which in particular means that all non-empty moduli space are smooth and of the expected dimension. It follows from Floer's gluing theorem applied to the broken trajectory 
$$\smash{(A_{L+1}; B_1, \ldots , B_J)}$$ 
that there is some ${\bf K}$-ASD connection in $\A^{1,p}(P; u_{L+1}^*a)$; see Remark \ref{FloerGluing}. Hence ${\mathcal{M}}_{ASD, {\bf K}}(u_{L+1}^*a)$ is non-empty, and must therefore have non-negative dimension. This proves the claim.

\medskip

This concludes our argument for long-time existence. Note that this same argument also excludes bubbling at infinite time. \qed

\subsection{Infinite-time convergence}\label{Infinite-TimeConvergence}

Here we complete the proof of Theorem \ref{heatflowtheorem} by showing that $A(\tau)$ converges, in the sense described, as $\tau \rightarrow \infty$. Of course, we may assume that $A(\tau)$ is not ${\bf K}$-YM for any $\tau$, since otherwise it would be constant in $\tau$ by uniqueness, and we would be done.

It follows from the analysis of the previous section that no bubbling can form along the flow. The argument also shows that the flow converges, modulo gauge, at infinite time to a \emph{broken} ${\bf K}$-ASD trajectory. In this section we refine this by showing the flow $A(\tau)$ converges at infinite time \emph{on all of $Z$} to an actual ${\bf K}$-ASD connection $A \in \A^{1,2}(P; a) \cap \A^{1,4}(P; a)$. We do this in several steps.

\begin{adjustwidth}{0pt}{}
\begin{customthmQuant}{Step 1}\label{Step1} 
$\sup_{\tau \geq 1} \Vert F_{A(\tau), {\bf K}} \Vert_{L^\infty(Z)} < \infty$.
\end{customthmQuant}
\end{adjustwidth}

\medskip

If such a uniform bound did not exist, then there would be a non-flat Yang-Mills bubble on $S^4$ (actually, it would have to be ASD since $\smash{F^+_{A(\tau)}}$ is $L^2$-small). However, as we saw in Section \ref{Proof2a}, the energy and index assumptions do not allow this to happen.

\begin{adjustwidth}{0pt}{}
\begin{customthmQuant}{Step 2}\label{Step2} 
\emph{Fix $s < \infty$. Then there is some $C_{PI} > 0$ so that the following Poincar\'{e} inequality holds for all $\tau \geq 0$:}
 \begin{equation}\label{pi2} 
\Vert F^+_{A(\tau), {\bf K}} \Vert_{L^s(Z)} \leq C_{PI} \Vert d_{\afA(\tau), {\bf K}} F^+_{A(\tau), {\bf K}} \Vert_{L^s(Z)} .
\end{equation}
\end{customthmQuant}
\end{adjustwidth}

\medskip

Fix $\tau$. Since $A(\tau)$ is not ${\bf K}$-YM, the existence of a constant $C = C(\tau)$ satisfying (\ref{pi2}) is obvious. It suffices to show that this constant does not diverge as $\tau$ approaches $\infty$. If this were the case, then we could find a sequence $\tau_n$ diverging to $\infty$ with
$$\Vert d_{\afA(\tau_n), {\bf K}} W_n \Vert_{L^s(Z)}  \rightarrow 0,$$
where we have set
$$W_n \defeq   F^+_{A(\tau_n), {\bf K}}  /  \Vert F^+_{A(\tau_n), {\bf K}} \Vert_{L^s(Z)}.$$
The proof now is very similar to that of Theorem \ref{epsdef1}. Namely, by Uhlenbeck's compactness theorem, it follows that there are gauge transformation $u_n$ so that the $u_n^* A(\tau_n)$ converge to a broken ${\bf K}$-YM trajectory 
$$(A^0; B^1, \ldots, B^J)$$
that is asymptotic to $u^*a$ for some gauge transformation $u$. We assume, for simplicity, that $u$ is the identity. Just as in Section \ref{Proof2a}, the index assumption on $a$ implies that this is actually a broken ${\bf K}$-ASD trajectory. In particular, by Floer's gluing theorem (see Remark \ref{FloerGluing}) we can find a sequence $A_n'$ of ${\bf K}$-ASD connections in $\A(P; a)$ that converge to this broken trajectory in the sense of Theorem \ref{UhlComp1}. Then Corollary \ref{UhlCor} implies 
$$\Vert  A(\tau_n) - A_n' \Vert_{L^\infty(Z)} \longrightarrow 0,$$
after possibly applying suitable gauge transformations to the $A_n'$. The assumptions on ${\bf K}$ imply that each of these $A_n'$ is ASD-regular, so by Remark \ref{uniformC}, there is a uniform constant $C$ so that
$$\begin{array}{rcl}
1 = \Vert W_n \Vert_{L^s(Z)} & \leq  & C \Vert d_{\afA_n', {\bf K}}  W_n \Vert_{L^s(Z)} \\
& \leq & C \Vert d_{\afA(\tau_n) , {\bf K}} W_n \Vert_{L^s(Z)} + C \Vert \afA(\tau_n) -  \afA_n' \Vert_{L^\infty(Z)}.
\end{array}$$
(Strictly speaking, Remark \ref{uniformC} is only stated for $s = 2$, but the estimate holds for all $s < \infty$ by standard Fredholm theory.) The right-hand side is going to zero, so this contradiction finishes the proof of \ref{Step2}. 

\begin{adjustwidth}{0pt}{}
\begin{customthmQuant}{Step 3}\label{Step3} 
\emph{The $A(\tau)$ converge exponentially in $L^2(Z)$ to some connection $A_\infty \in \A^{0,2}(P; a)$.}
\end{customthmQuant}
\end{adjustwidth}

\medskip

Use the flow (\ref{ymheatflow}), together with the estimate (\ref{pi2}) to get 

$$\begin{array}{rcccl}
\frac{d}{d\tau} \Vert F^+_{\afA(\tau), {\bf K}} \Vert^2_{L^2} & = & -2 ( d_{\afA, {\bf K}}^+ d_{\afA, {\bf K}}^* F_{\afA, {\bf K}}, F_{\afA, {\bf K}}^+) & = &   -4 ( d_{\afA, {\bf K}}^+ d_{\afA, {\bf K}}^* F_{\afA, {\bf K}}^+, F_{\afA, {\bf K}}^+)\\
&&\\
& = & -4 \Vert d_{\afA, {\bf K}}^* F_{\afA, {\bf K}}^+\Vert_{L^2}^2 & \leq & -4 C_{PI}^{-2} \Vert F_{\afA, {\bf K}}^+ \Vert_{L^2}^2.
\end{array}$$
This implies exponential convergence of $F^+_{\afA, {\bf K}}$ to zero: 

\begin{equation}\label{expconve}
\Vert F^+_{\afA(\tau), {\bf K}} \Vert_{L^2(Z)}^2 \leq B^2 e^{-4\tau /C_{PI}^2},
\end{equation}
where we have set
$$B \defeq \Vert F^+_{\afA_0, {\bf K}} \Vert_{L^2}.$$

Next, integrate (\ref{theflowgibes}) over an interval $\left[\tau_a, \tau_b \right]$ to get 
$$\intdd{\tau_a}{\tau_b} \:  \Vert d_{\afA, {\bf K}}^* F_{\afA, {\bf K}} \Vert_{L^2}^2 = \Vert F_{\afA(\tau_a), {\bf K}}^+ \Vert^2_{L^2} - \Vert F_{\afA(\tau_b), {\bf K}}^+ \Vert^2_{L^2} \leq   B^2 e^{-4\tau_a /C_{PI}^2}.$$
Combining this with (\ref{beforeclaim}) gives 
$$\Vert A(\tau_b) - A(\tau_a) \Vert_{L^2(Z)} \leq \vert \tau_b - \tau_a \vert  B e^{-2\tau_a /C_{PI}^2},$$
which holds for all $0\leq \tau_a \leq \tau_b$. Fix $\tau \geq 0$ and use the above repeatedly with $\tau_a = \tau+j$ and $\tau_b = \tau+j+1$, for $0\leq j \leq J-1$, to get

$$\begin{array}{rcl}
\Vert \afA(\tau + J) - \afA(\tau ) \Vert_{L^2(Z)}  & \leq & \sumdd{j = 0}{J-1} \Vert \afA(\tau +j +1) - \afA(\tau + j) \Vert_{L^2(Z)} \\
& \leq & B e^{-2 \tau / C_{PI}^2} \sumdd{j=0}{J - 1} e^{-2 j/ C_{PI}^2}\\
& \leq  & \frac{B }{1 - e^{-2/C_{PI}^2}} e^{-2 \tau / C_{PI}^2}.
\end{array}$$
This type of argument shows that $\left\{\afA(\tau)\right\}_{\tau}$ is Cauchy and so converges in $L^2$ to some limiting connection $\afA_\infty$ with the same $K$-flat limits as the $\afA(\tau)$. This argument also shows exponential convergence in $L^2$:

\begin{equation}\label{ratherpreciseconvergenceestimate}
\Vert \afA(\tau) - \afA_\infty \Vert_{L^2(Z), g}   \leq \frac{\Vert F^+_{\afA_0, {\bf K}} \Vert_{L^2(Z), g} }{1 - e^{-2/C_{PI}^2}}  e^{-2 \tau / C_{PI}^2}, \indent \tau \geq 0.
\end{equation}
This completes the proof of \ref{Step3}.
%

\medskip

Our goal now is to bootstrap from $L^2$-convergence to the higher Sobolev convergence claimed in the statement of Theorem \ref{heatflowtheorem}. This will call for higher order versions of the estimates above. To simplify the discussion, for the rest of the proof we treat the perturbation ${\bf K}$ as being zero, and drop it from the notation. \ref{axioms1} provides the estimates necessary to extend these arguments to the case where ${\bf K}$ is non-zero.

\begin{adjustwidth}{0pt}{}
\begin{customthmQuant}{Step 4}\label{Step4} 
\emph{For each $2 \leq s \leq 4$, there is some $C$ so that the following holds for all $\tau \geq 1$:}
 $$\Vert d_{A(\tau)} F^+_{A(\tau)} \Vert_{L^s(Z)} \leq C\Vert d_{\afA(\tau)}^+ d_{\afA(\tau)} F^+_{A(\tau)} \Vert_{L^2(Z)} .$$
\end{customthmQuant}
\end{adjustwidth}

\medskip

We will prove the statement for $s = 4$. The proof for $s = 2$ is similar, and the result for general $2 \leq s \leq 4$ follows by interpolation. 

Our strategy is to use a contradiction argument similar to the one in \ref{Step2}. That is, we assume there is a sequence $\tau_n$ diverging to $\infty$ with
$$\Vert d_{A(\tau_n)}^+ d_{A(\tau_n)}^* W_n \Vert_{L^2(Z)} \longrightarrow 0,$$
where we have set
$$W_n \defeq  F_{A(\tau_n)}^+ / \Vert d_{A(\tau_n)} F^+_{A(\tau_n)} \Vert_{L^4(Z)}.$$
As in \ref{Step2}, it follows from Uhlenbeck's compactness theorem and Theorem \ref{UhlCor} that there are ASD connections $A_n' \in \A(P; a)$ so that 
$$\limd{n \rightarrow \infty} \Vert A(\tau_n) -  A_n' \Vert_{L^\infty(Z)} + \Vert  A(\tau_n) - A_n' \Vert_{W^{1,4}(Z)}  = 0.$$
Then we have
$$d_{A(\tau_n)} W_n = d_{A_n'} W_n + \left[ A(\tau_n) - A_n' \wedge  W_n \right].$$
Thus
\begin{equation}\label{155qqrwe}
\begin{array}{rcl}
1 = \Vert d_{A(\tau_n)} W_n  \Vert_{L^4(Z)} & \leq & \Vert d_{A_n'} W_n  \Vert_{L^4(Z)} + \Vert A(\tau_n) - A_n' \Vert_{L^\infty(Z)} \Vert W_n \Vert_{L^4(Z)}\\
&&\\
& \leq & \Vert d_{A_n'}  W_n  \Vert_{L^4(Z)} + C_{PI} \Vert A(\tau_n) - A_n' \Vert_{L^\infty(Z)},
\end{array}
\end{equation}
where $C_{PI}$ is the constant from \ref{Step2} with $s = 4$. 

Integration by parts shows that $d_{A}^+ $ is always injective on the image of $*d_{A}\vert_{\Omega^+}$. Indeed, suppose $W$ is a self-dual 2-form with $d_A^+*  d_AW = 0$. Then 
$$0 = (d_A^+ * d_A W, W) = (d_A * d_A W, W) = (* d_A W, d_A^* W) = - \Vert d_A W \Vert^2_{L^2(Z)}.$$

In particular, by the Fredholm property for $d_{A_n}$, there is some constant $C$ so that
$$\Vert d_{A_n'} W  \Vert_{L^4(Z)}  \leq C \Vert d_{A_n'}^+ * d_{A_n'} W  \Vert_{L^2(Z)}$$
for all self-dual 2-forms $W$. Moreover, the constant $C$ can be taken to be independent of $A_n'$, since the $A_n'$ are ASD with uniformly bounded energy; see Remark \ref{uniformC}. Combining this with (\ref{155qqrwe}) gives
$$\begin{array}{rcl}
1 & \leq & C \Vert d_{A_n'}^+ * d_{A_n'} W_n  \Vert_{L^2(Z)} + C_{PI} \Vert A(\tau_n) - A_n' \Vert_{L^\infty(Z)}\\
&& \\
& \leq & C \Vert d_{A(\tau_n)}^+ * d_{A(\tau_n)} W_n \Vert_{L^2(Z)} + C_{PI} \Vert A(\tau_n) - A_n' \Vert_{L^\infty(Z)}\\
&& + 2C \Vert A_n' - A(\tau_n) \Vert_{L^4(Z)} \left( \Vert W_n \Vert_{L^4(Z)} + \Vert d_{A_n'} W_n \Vert_{L^4(Z)} \right)\\
&& + 2C \Vert A_n' - A(\tau_n) \Vert_{W^{1,4}(Z)} \Vert W_n \Vert_{L^4(Z)}.
\end{array}$$
In the second line we used
$$d_{A'} * d_{A'} W = d_{A} * d_{A} W  - d_{A'} \left( * \left[ A'- A \wedge  W \right] \right) + \left[ A' - A \wedge d_{A} W \right] ,$$
together with H\"{o}lder estimates. The $L^4$-norms of $W_n $ and $d_{A_n'} W_n $ are uniformly bounded, so we can continue the above as
$$1 \leq C' \left( \Vert d_{A(\tau_n)}^+ * d_{A(\tau_n)} W_n \Vert_{L^2(Z)} + \Vert A(\tau_n) - A_n' \Vert_{L^\infty(Z)} +  \Vert A_n' - A(\tau_n)  \Vert_{W^{1,4}(Z)} \right).$$
The right-hand side is going to zero. This contradiction establishes \ref{Step4}.

\begin{adjustwidth}{0pt}{}
\begin{customthmQuant}{Step 5}\label{Step5} 
\emph{The connections $A(\tau)$ converge exponentially in $L^4(Z)$ to $A_\infty$.}
\end{customthmQuant}
\end{adjustwidth}

\medskip

We will show that $A(\tau)$ is Cauchy in $\A^{0,4}(P; a)$ as $\tau$ approaches $\infty$. Ultimately we will establish a second-order version of \ref{Step3} (the full second order version is \ref{Step6}, below). All unspecified norms are $L^2(Z)$. 

We begin by establishing some preliminary estimates. Differentiating and using the flow gives
$$\begin{array}{rcl}
\fracd{d}{d\tau} \frac{1}{2} \Vert d_A F^+_A \Vert^2 & = & - (d_A d_A^+ d_A^* F_A, d_A F^+_A )\\
&& - \left( \left[ d_A^* F_A \wedge F_A^+ \right], d_A F_A^+ \right).
\end{array}$$
Integrate by parts in the first term on the right, and use $d^*_A F_A = -2 * d_A F^+_A$ to get
$$  - (d_A d_A^+ d_A^* F_A, d_A F^+_A ) = - (d_A^+ * d_A F_A^+, * d_A * d_A F_A^+).$$
The quantity $d_A^+ * d_A F_A^+$ is the self-dual part of $*d_A * d_A F_A^+$, and the splitting $ \Omega^2 = \Omega^+ \oplus \Omega^-$ is an $L^2$-orthogonal decomposition. This allows us to write
$$  - (d_A d_A^+ d_A^* F_A, d_A F^+_A )  = - 2 \Vert d^+_A * d_A F_A^+ \Vert^2.$$
Putting these together gives
\begin{equation}\label{ASDHigherDer}
\begin{array}{rcl}
\fracd{d}{d\tau} \frac{1}{2} \Vert d_A F^+_A \Vert^2 & = & - 2 \Vert d^+_A * d_A F_A^+ \Vert^2  +2  \left( \left[ * d_A F^+_A \wedge F_A^+ \right], d_A F_A^+ \right).
\end{array}
\end{equation}
Use H\"{o}lder's inequality on the cubic term to get
$$\begin{array}{rcl}
\fracd{d}{d\tau} \frac{1}{2} \Vert d_A F^+_A \Vert^2 & \leq &-  \Vert d^+_A * d_A F_A^+ \Vert^2 -  \Vert d^+_A * d_A F_A^+ \Vert^2 + 4 \Vert F_A^+ \Vert \Vert d_A F^+_A \Vert^2_{L^4(Z)}\\
&&\\
& \leq & - c \Vert d_A F_A^+ \Vert^2 - c \Vert d_A F_A^+ \Vert_{L^4(Z)}^2 + 4 \Vert F_A^+ \Vert \Vert d_A F^+_A \Vert_{L^4(Z)}^2.
\end{array}$$
We used \ref{Step4} in the second line with $s = 2$, and again with $s = 4$. We know that $ \Vert F_A^+ \Vert $ is converging to zero, so the sum of the last two terms above becomes negative. Hence 
$$\fracd{d}{d\tau} \frac{1}{2} \Vert d_A F^+_A \Vert^2  \leq - c \Vert d_A F_A^+ \Vert^2 $$
for all sufficiently large $\tau$. This implies exponential decay
$$\Vert d_{A(\tau)} F^+_{A(\tau)} \Vert^2  \leq C e^{-2c \tau }.$$

Next, integrate (\ref{ASDHigherDer}) over an interval $\left[ \tau_a, \tau_b \right]$, with $\tau_a$ large, to get
$$\begin{array}{rcl}
\intdd{\tau_a}{\tau_b} \Vert d_A^+ * d_A F_A^+ \Vert^2 & = & \frac{1}{2} \Vert d_{A(\tau_a)} F_{A(\tau_a)}^+ \Vert^2 - \frac{1}{2} \Vert d_{A(\tau_b)} F_{A(\tau_b)}^+ \Vert^2\\
&& -\intdd{\tau_a}{\tau_b} \left( \left[ d_A^* F_A \wedge F_A^+ \right], d_A F_A^+ \right)\\
&&\\
& \leq & \frac{1}{2} \Vert d_{A(\tau_a)} F_{A(\tau_a)}^+ \Vert^2 \\
&&+  2\left( \sup_{\tau} \Vert F_{A(\tau)}^+ \Vert \right) \intdd{\tau_a}{\tau_b} \Vert d_A F_A^+ \Vert_{L^4(Z)}^2\\
&&\\
& \leq & \frac{1}{2} \Vert d_{A(\tau_a)} F_{A(\tau_a)}^+ \Vert^2 +  \frac{1}{2}  \intdd{\tau_a}{\tau_b} \Vert d^+ * d_A F_A^+ \Vert^2.
\end{array}$$
The last inequality used \ref{Step4} to estimate the $L^4$-norm; we also used that $\Vert F_{A(\tau)}^+ \Vert $ converges to zero, and so this holds provided $\tau_a$ is sufficiently large. We therefore have
$$\intdd{\tau_a}{\tau_b} \Vert d_A^+ * d_A F_A^+ \Vert^2 \leq \Vert d_{A(\tau_a)} F_{A(\tau_a)}^+ \Vert^2 \leq  C e^{-2c \tau_a}.$$

Now we verify \ref{Step5}. Take the $L^4$-norm of (\ref{ab}) to get
$$\begin{array}{rcl}
\Vert A(\tau_b) - A(\tau_a) \Vert_{L^4(Z)}^2 & \leq & \vert \tau_b - \tau_a \vert  \intdd{\tau_a}{\tau_b} \: \Vert d_{A}^* F_{A} \Vert^2_{L^4(Z)} \: d\tau.
\end{array}$$
By \ref{Step4} and the above exponential decay, we can continue this as follows
$$\begin{array}{rcl}
\Vert A(\tau_b) - A(\tau_a) \Vert_{L^4(Z)}^2 & \leq & C' \vert \tau_b - \tau_a \vert  \intdd{\tau_a}{\tau_b} \: \Vert d^+_A * d_{A} F^+_{A} \Vert^2 \: d\tau\\
&&\\
& \leq & C'' \vert \tau_b - \tau_a \vert e^{-2c \tau_a},
\end{array}$$
provided $\tau_a$ is sufficiently large. As in \ref{Step3}, this implies that $A(\tau)$ is Cauchy in $L^4(Z)$, and converges exponentially in $L^4(Z)$; the limit is necessarily $A_\infty$.

\begin{adjustwidth}{0pt}{}
\begin{customthmQuant}{Step 6}\label{Step6} 
\emph{The connections $A(\tau)$ converge exponentially in $W^{1,2}(Z)$ to $A_\infty$.}
\end{customthmQuant}
\end{adjustwidth}

\medskip

All unspecified norms are $L^2(Z)$. Recall the Sobolev norms are defined relative to the fixed reference connection $A_{ref}$. Therefore
$$\begin{array}{rcl}
\Vert A - A_{\infty} \Vert_{W^{1,2}(Z)} & \leq  & C \left( \Vert d^+_{A_{ref}} (A - A_\infty) \Vert  + \Vert d^*_{A_{ref}} (A - A_\infty)  \Vert + \Vert A -A_\infty \Vert \right)\\
&&\\
& \leq & C' \Big( \Vert d^+_{A_\infty} (A - A_\infty) \Vert  + \Vert d^*_{A_\infty} (A - A_\infty)  \Vert + \Vert A -A_\infty \Vert \Big. \\
&& \Big. +  \Vert A_{\infty} - A_{ref} \Vert_{L^4(Z)} \Vert A - A_\infty \Vert_{L^4(Z)}\Big)
\end{array}$$
We know that $A(\tau)$ converges to $A_\infty$ in $L^2(Z) \cap L^4(Z)$, so to establish \ref{Step6}, it suffices to show
$$\limd{\tau \rightarrow \infty} \Vert d^+_{A_\infty} (A(\tau) - A_\infty) \Vert  + \Vert d^*_{A_\infty} (A(\tau) - A_\infty)  \Vert = 0.$$
We will work with the first limit; the other is similar (use the second Bianchi identity). Apply $d^+_{A_\infty}$ to both sides of (\ref{ab}) to get
\begin{equation}\label{didifindit}
\begin{array}{rcl}
d^+_{A_\infty}(A(\tau) - A_\infty) & = & - d^+_{A_\infty}  \intdd{\tau}{\infty} d_A^* F_A \: d\tau \\
&&\\
& = & -  \intdd{\tau}{\infty} d^+_{A_\infty}  d_A^* F_A \: d\tau.
\end{array}
\end{equation}

\begin{remark}\label{morepre}
The justification for the exchange of the derivative and the integral is as follows: The proof of Theorem \ref{Short-TimeExistence} shows that there is a smooth path $B$ of connections, and a path of gauge transformation of class  
$$u \in \C^1((0, \infty), \G^{1,p}(P; e) \cap \G^{1,2}(P; e))$$
so that $u^*A = B$. Here $ p > 4$ is as in the statement of Theorem \ref{heatflowtheorem}. Then
$$d^+_{A_\infty}  d_A^* F_A = \mathrm{Ad}(u) d^+_{(u^{-1})^* A_{\infty}} d_A^* F_A,$$
which is uniformly bounded in $L^2(Z)$.
\end{remark}
 
Using $d_A^* F_A = - 2*d_A F_A^+$, and converting the $d_{A_\infty}$ into $d_A$, the equalities in (\ref{didifindit}) become
$$d^+_{A_\infty}(A(\tau) - A_\infty) =  2 \intdd{\tau}{\infty} d_A^+ * d_A F^+_A + 2 \left[ A_\infty - A \wedge  d_A* F^+_A  \right] \: d\tau$$
Next, take the $L^2$-norm of both sides and use H\"{o}lder's inequality
$$\begin{array}{rcl}
\Vert d_{A_\infty}^+(A(\tau) - A_\infty)  \Vert  & \leq & 2\intdd{\tau}{\infty} \Vert d^+_A * d_A F^+_A \Vert \\
&& \indent + \Vert A_\infty - A \Vert_{L^4(Z)} \Vert d_A  F^+_A \Vert_{L^4(Z)} \: d\tau\\
&&\\
& \leq & C \intdd{\tau}{\infty} \Vert d_A^+ * d_A F^+_A \Vert+ \Vert d_A F^+_A \Vert \: d\tau.
\end{array}$$
Here we used \ref{Step4} to convert away from the $L^4$-norm on $d_A F^+_A$, and we used \ref{Step5}, which gives a uniform bound on $\Vert A_\infty - A(\tau) \Vert_{L^4(Z)}$. Next, we have
$$\begin{array}{rcl}
\Vert d^+_{A_\infty}(A(\tau) - A_\infty)  \Vert  & \leq & C \sumdd{j = 0}{\infty} \intdd{\tau + j }{j + 1} \Vert d^+_A * d_A F^+_A \Vert_{L^2(Z)} + \Vert d_A F^+_A \Vert_{L^2(Z)} \: d\tau\\
&&\\
& \leq & C \sumdd{j = 0}{\infty} \left( \intdd{\tau + j }{ \tau + j + 1} \Vert d_A^+ * d_A F^+_A \Vert_{L^2(Z)}^2  +\Vert d_A F^+_A \Vert_{L^2(Z)}^2 \: d\tau \right)^{1/2}\\
&&\\
& \leq & C' \sumdd{j = 0}{\infty} e^{-c (\tau+j)/2}.
\end{array}$$
The last inequality follows by the exponential convergence that was established in \ref{Step5}.

\begin{adjustwidth}{0pt}{}
\begin{customthmQuant}{Step 7}\label{Step7} 
\emph{For $1 \leq s < \infty$, the $A(\tau)$ converge exponentially in $L^s(Z)$ to $A_\infty$.}
\end{customthmQuant}
\end{adjustwidth}

\medskip

Since $1 \leq s < \infty$, we have $2 \leq 4s/(s+1) < 4$, and so
$$W^{1,2}(Z) \cap W^{1,4}(Z) \subset W^{1,4s/(s+1)} (Z) \subset L^s(Z)$$
These embeddings involve constants that are independent of the connection used to define the Sobolev norms, and are valid even though $Z$ is non-compact. As we saw earlier, $d_A^+$ is injective on the image of $* d_A \vert_{\Omega^+}$. Together with the embeddings above, it follows that there is a bound of the form
$$\Vert d_A W \Vert_{L^s(Z)} \leq C \left( \Vert d_A^+ * d_A W \Vert_{L^2(Z)} + \Vert d_A^+ * d_A W \Vert_{L^4(Z)}  \right)$$
for all self-dual 2-forms $W$, and all connections $A = A(\tau)$ along the flow. We also have that $d_A$ is injective on the image of $d_A^+$, and so we can convert away from the $L^4$-norm to get
$$\Vert d_A W \Vert_{L^s(Z)} \leq C' \left( \Vert d_A^+ * d_A W \Vert_{L^2(Z)} + \Vert d_A d_A^+ * d_A W \Vert_{L^2(Z)}  \right)$$
for a constant that is independent of $W$ and $\tau$. 

Now turning to the problem at hand, take the norm of (\ref{ab}), and use the above $L^s$-estimates to get
$$\begin{array}{rcl}
\Vert A(\tau) - A_\infty \Vert_{L^s(Z)} & \leq & 2 \intdd{\tau}{\infty } \Vert d_A F^+_A \Vert_{L^s(Z)} \: d\tau \\
&&\\
& \leq & C' \intdd{\tau}{\infty}   \Vert d_A^+ * d_A F_A^+ \Vert_{L^2(Z)} +  \Vert d_A d_A^+ * d_A F_A^+ \Vert_{L^2(Z)} \: d \tau.
\end{array}$$
Now we want to estimate the integrand. Upon differentiating, one finds
$$\begin{array}{rcl}
\frac{d}{d\tau} \frac{1}{2} \Vert d_A^+ * d_A F_A^+ \Vert_{L^2(Z)}^2 & = & - 2 \Vert d_A d_A^+ * d_A F_A^+ \Vert^2_{L^2(Z)}\\
&& + 2 \left( \left[ * d_A F_A^+ \wedge * d_A F_A^+ \right], d_A^+ * d_A F_A^+ \right)\\
&& + 2 \left( d_A * \left[ * d_A F_A^+ \wedge F_A^+ \right], d_A^+ * d_A F_A^+ \right).
\end{array}$$
The terms that are cubic in $F_A^+$ can be estimated as we did in \ref{Step5}. Just as in that step, this gives exponential decay for $\Vert d_A^+ * d_A F_A^+ \Vert_{L^2(Z)}^2 $. Then it also gives exponential decay for 
$$\intdd{\tau}{\infty} \Vert d_A d_A^+ * d_A F_A^+ \Vert_{L^2(Z)} \: d \tau,$$
which finishes the argument.

\begin{adjustwidth}{0pt}{}
\begin{customthmQuant}{Step 8}\label{Step8} 
\emph{For $2 \leq q \leq 4$, the $A(\tau)$ converge exponentially in $W^{1,q}(Z)$ to $A_\infty$.}
\end{customthmQuant}
\end{adjustwidth}

\medskip

We will establish this for $q = 4$; the remaining values of $q$ follow by interpolation and \ref{Step6}.

The analysis is very similar to \ref{Step6}, so we will be brief. Begin by writing
$$\begin{array}{rcl}
\Vert A - A_\infty \Vert_{W^{1,4}(Z)} & \leq & C \Big( \Vert d_{A_\infty}^+ (A- A_\infty) \Vert_{L^4(Z)} + \Vert d_{A_\infty}^* (A- A_\infty) \Vert_{L^4(Z)} \Big. \\
&&  \Big. +  \Vert A - A_\infty \Vert_{L^4(Z)} + \Vert A - A_{ref} \Vert_{L^8(Z)}  \Vert A - A_\infty \Vert_{L^8(Z)} \Big).
\end{array}$$
By \ref{Step7}, it suffices to show the first two terms on the right go to zero. We will focus on the first of these. Take the $L^4$-norm of both sides of (\ref{didifindit}) to get
$$\begin{array}{rcl}
\Vert d^+_{A_\infty} (A(\tau) - A_\infty) \Vert_{L^4(Z)} & \leq &  2 \intdd{\tau}{\infty} \Vert d_A^+ * d_A F_A^+ \Vert_{L^4(Z)} \\
&& \indent \indent \Vert A - A_\infty \Vert_{L^8(Z)} \Vert d_A F_A^+ \Vert_{L^8(Z)}\\
&&\\
& \leq &  C \intdd{\tau}{\infty} \Vert d_A^+ * d_A F_A^+ \Vert_{L^2(Z)} + \Vert d_A d_A^+ * d_A F_A^+ \Vert_{L^2(Z)} \\
&& \indent \indent \Vert A - A_\infty \Vert_{L^8(Z)} \Vert d_A F_A^+ \Vert_{L^8(Z)},
\end{array}$$
where we used Lemma \ref{l2} to estimate the $L^4$-norm. As in \ref{Step7}, the right-hand side goes to zero exponentially. \qed


\end{document}